\documentclass[11pt]{article}
\usepackage[T1]{fontenc}
\usepackage{amsmath,amsthm,amsfonts,amssymb,mathrsfs,bm}
\usepackage{amsmath,amsthm,amsfonts,amssymb,bm,wasysym}
\usepackage{epsfig}
\usepackage[usenames]{color}
\usepackage[usenames,dvipsnames]{xcolor}
\usepackage{verbatim}
\usepackage{hyperref}
\usepackage{multicol}
\usepackage{comment}
\usepackage{float}
\usepackage{graphicx}
\usepackage{centernot}
\usepackage{tikz}
\usepackage{color}
\usepackage[]{algorithm2e}
\usepackage{enumerate}
\usepackage{comment}
\usepackage[capitalize]{cleveref}
\usepackage{fontenc}
\usepackage{caption}

\usepackage{hyperref}
\hypersetup{
	colorlinks   = true, 
	urlcolor     = red, 
	linkcolor    = blue, 
	citecolor   = blue 
}

\Crefname{figure}{Figure}{Figures}

\parskip \medskipamount
\newtheorem{theorem}{Theorem}[section]
\newtheorem{definition}[theorem]{Definition}

\numberwithin{equation}{section}
\newtheorem{lemma}[theorem]{Lemma}
\newtheorem{proposition}[theorem]{Proposition}
\newtheorem{corollary}[theorem]{Corollary}

\newtheorem{claim}[theorem]{Claim}
\newtheorem*{claim*}{Claim}

\newtheorem{question}[theorem]{Question}
\newtheorem{conjecture}[theorem]{Conjecture}
\numberwithin{equation}{section}

\newcommand{\eps}{\varepsilon}
\newcommand{\union}{\cup}

\newcommand{\floor}[1]{\lfloor{#1}\rfloor}
\newcommand{\ceil}[1]{\lceil{#1}\rceil}
\newcommand{\abs}[1]{\left|{#1}\right|}

\newcommand{\chorddiagrams}[2]{{\mathcal{X}_{#1, #2}}}
\newcommand{\chordsignatures}[1]{{\mathcal{X}_{#1}}}
\newcommand{\surfacemaps}[2]{{\mathcal{M}_{#1, #2}}}
\newcommand{\parallelsigs}[2]{{\mathcal{S}_{#1, #2}}}

\definecolor{bittersweet}{rgb}{1.0, 0.44, 0.37}

\newcommand{\captionstyle}[1]{{#1}}

\title{Polynomial mixing for polygonal side matchings}
\author{Renan Gross\thanks{DPMMS, University of Cambridge. rg751@cam.ac.uk} \and An{\dj}ela \v{S}arkovi\'c \thanks{King's College and DPMMS, University of Cambridge. as2572@cam.ac.uk}}

\begin{document}

\maketitle

\begin{abstract}
	We introduce a natural Markov chain on chord diagrams, which, at every step, selects two random chords and swaps them if doing so preserves the diagram's genus. This generalizes the chord swap chain on the Catalan structure of non-intersecting chord diagrams. We show that for fixed genus, the chain mixes in polynomial time.
\end{abstract}
	
\section{Introduction}

Let $n >0$ be an integer, and let $P$ be a regular polygon with $2n$ vertices, labelled clockwise from $1$ to $2n$. A chord is a matching between two sides of $P$, and a chord diagram $X$ is a set of $n$ chords which constitute a perfect matching of the sides of $P$. 

The genus $g$ of a chord diagram $X$ is the genus of the orientable surface $S$ obtained by identifying the sides of $P$ according to the chords of $X$ while maintaining the orientation (see \cref{fig:diagram_and_surface}). The genus can be calculated using Euler's formula: treating the boundary of the original polygon as a planar graph with $2n$ vertices, $2n$ edges and two faces, after identifying the sides we are left with an $S$-embedded graph with $V$ vertices, $n$ edges, and one face; the genus is then given by 
\begin{equation} \label{eq:euler_characteristic}
	g=\frac{2+E-F-V}{2}=\frac{n+1-V}{2} \, 
\end{equation}
(see e.g. \cite[Chapter 1]{lando_zvonkin_graphs_on_surfaces} for a reference on graphs and surfaces). Since $V \geq 1$, the maximum possible genus is $n/2$ (for even $n$). 

\begin{figure}[H] 
	\centering
	\includegraphics[width=0.8\textwidth]{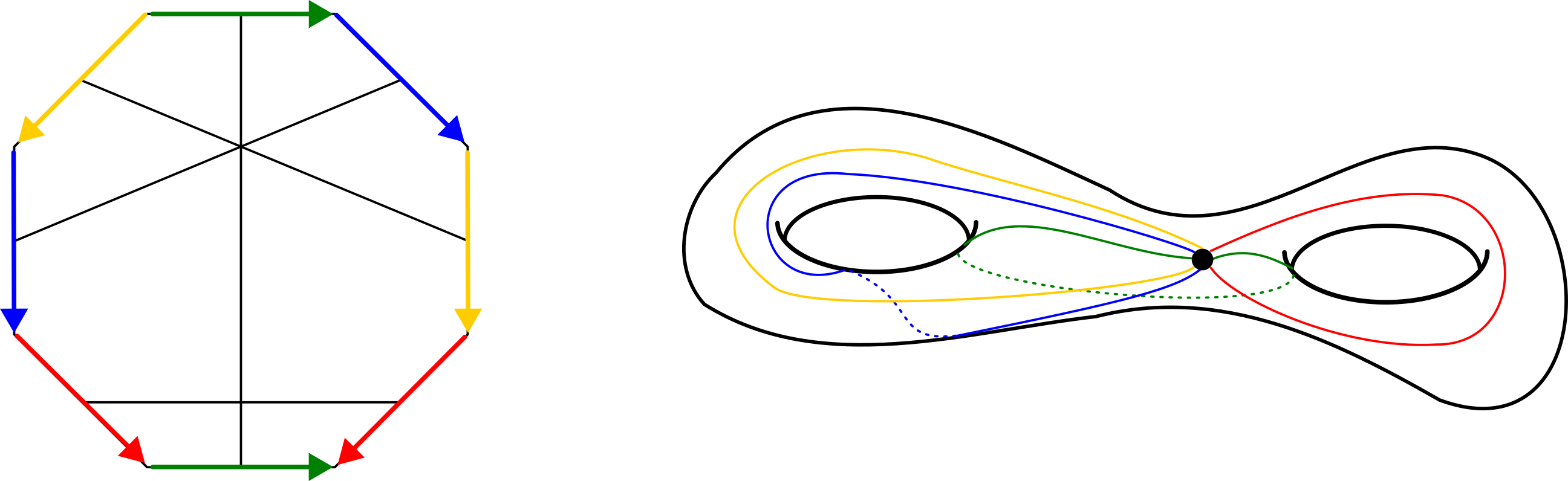}
	\captionsetup{width=.8\linewidth}
	\caption{\captionstyle{\textbf{Left}: a genus $2$ chord diagram on the octagon. \textbf{Right}: a cartoon of the surface obtained by gluing the sides together, together with an embedding of the glued-up octagon.}}
	\label{fig:diagram_and_surface}
\end{figure}	

We say that $Y$ is a chord swap of $X$ if there exist two distinct chords $\{A,B\}, \{C,D\} \in X$ such that $Y$ is equal to $X$ but with the chord endpoints swapped, i.e. 
\begin{equation*}
	Y = \left( X \setminus \{ \{A,B\}, \{C,D\}\}\right) \union \{ \{A,C\}, \{B,D\}\} 
\end{equation*}
or 
\begin{equation*}
	Y = \left( X \setminus \{\{A,B\}, \{C,D\}\}\right) \union \{\{A,D\}, \{B,C\}\} \, .
\end{equation*}

In this paper, we investigate a chord-swapping Markov chain on chord diagrams of fixed genus.

\begin{definition}\label{def:fixedgenuschain} 
	For positive integers $n \geq 2$ and $g \leq n/2$, let $\chorddiagrams{n}{g}$ be the set of all genus $g$ chord diagrams with $n$ chords. The \emph{fixed genus chord swap chain} on $\chorddiagrams{n}{g}$ is the Markov chain obtained by the following transition kernel:
	\begin{equation*}
		P(X,Y) = \begin{cases}
			\frac{1}{4 {n \choose 2}} & Y \text{ is a chord swap of $X$ and has the same genus as $X$} \\
			1 - \sum_{Z \in \chorddiagrams{n}{g} \backslash \{X\}} P(X,Z) & X=Y \, .
		\end{cases}
	\end{equation*}
\end{definition}
This Markov chain can be obtained by the following procedure: with probability $1/2$, do nothing; otherwise, if the current state is the chord diagram $X$, choose two chords $\{A,B\}$ and $\{C,D\}$ uniformly from all possible pairs of chords, and let $X'$ be a uniformly random chord diagram selected from the two possible chord swaps. The next state is $X'$ if it has the same genus as $X$; otherwise, the chain stays as $X$.

The transition matrix is symmetric and so the chain is reversible and its stationary distribution $\pi(x)$ is the uniform distribution on $\chorddiagrams{n}{g}$. However, it is not immediately apparent from the definition that it is irreducible; we will show that this is indeed the case when $n > 2g$ (and so the chain converges to its stationary distribution), and that when $n=2g$ there is one isolated diagram. In fact, we show that the diameter of the Markov chain graph is bounded above by $C(n + g^2)$ for some universal constant $C > 0$ (see \cref{sec:irreducibility}).

Our main result is that for every fixed $g > 0$, the fixed genus edge swap chain has spectral gap polynomial in $n$. 
\begin{theorem} \label{thm:spectral_gap_lower_bound}
	Let $g > 0$, let $n > 2g$, and let $\gamma_{n,g}$ be the spectral gap of the fixed genus chord swap chain on $\chorddiagrams{n}{g}$. There exists a constant $C_g > 0$ depending only on the genus such that 
	\begin{equation*}
		\gamma_{n,g} \geq C_g n^{-1488g+616} \, .
	\end{equation*}
\end{theorem}
This theorem gives a polynomial-in-$n$ mixing time bound: since $\abs{\chorddiagrams{n}{g}} \leq (2n-1)!! \leq (2n)^{2n}$, using the well-known relation $t_\mathrm{mix}(\eps) \leq \ceil{\frac{1}{\gamma}\log\left(\frac{1}{\eps \min_x \pi(x)} \right)}$, we have 
\begin{corollary}
	The mixing time of the fixed genus chord swap chain satisfies 
	\begin{equation*}
		t_\mathrm{mix}(\eps) \leq C_g n^{1488g-615}(\log n + \log\frac{1}{\eps}) \, .
	\end{equation*}
\end{corollary}
Two chords are said to \textit{intersect} or \textit{cross} each other if the straight lines connecting the centres of the sides intersect each other. It is not hard to show that a chord diagram has genus $0$ if and only if its chords do not cross. Our Markov chain is therefore a generalization to higher genus of a chord swapping chain on non-intersecting chord diagrams, a Catalan object of special interest (for example, it was proposed by McShine and Tetali as a possible comparison chain for chord-flips on triangulations \cite{mcshine_tetali_mixing_triangulation}). In her thesis, Cohen showed the following bound on the spectral gap.
\begin{theorem}[Application of Proposition 2.11 in \cite{cohen_thesis}] \label{thm:cohem_gap_lower_bound}
	Let $\gamma_{n,0}$ be the spectral gap of the non-intersecting chord swap chain. There exists a universal constant $C >0$ such that 
	\begin{equation}  \label{eq:cohen_gap_lower_bound}
		\gamma_{n,0} \geq C n^{-4} \, .
	\end{equation}
\end{theorem}
There is a large discrepancy between Cohen's $n^{-4}$ spectral gap lower bound for genus $0$ and \cref{thm:spectral_gap_lower_bound}'s $n^{-872}$ spectral gap lower bound for genus $1$, and we do not believe that the coefficient in front of $g$ in the exponent in \cref{thm:spectral_gap_lower_bound} is tight. In part, this is because we did not try to optimize it; but even if we had, we do not believe the methods used in the proof (restriction-projection; see \cref{sec:spectral_gap_lower_bound}) could produce optimal constants. In fact, even for genus $0$, the asymptotics of $\gamma_{n,0}$ are not known: the best upper bound we know of, due to private correspondence with Alessandra Caraceni, is
\begin{equation*}
	\gamma_{n,0} \leq C n^{-5/2} \, .
\end{equation*} 
The bound is obtained by observing how a chord swap affects the height function $f$ of the plane tree corresponding to the genus $0$ chord diagram: a single chord swap can decrease the height by at most $1$, and this can only occur when a swapped chord is on the path to the furthest leaf. The Dirichlet form of $f$ is therefore proportional to $\mathbb{E}_\pi[f] \cdot 1/n^2$. Since $\mathbb{E}_\pi[f] \asymp \sqrt{n}$ and since $\mathrm{Var}(f) \asymp n$ \cite[Proposition VII.16]{flajolet_sedgewick_analytic_combinatorics}, the Rayleigh quotient is therefore of the order 
\begin{equation*}
	\mathcal{R}(f) = \frac{\mathcal{E}(f)}{\mathrm{Var}(f)} \asymp \frac{\sqrt{n}/n^2}{n} = n^{-5/2} \, .
\end{equation*}
Finding the exact asymptotics of $\gamma_{n,g}$ for general $n$ and $g$ therefore currently seems out of reach. In fact, while \cref{thm:spectral_gap_lower_bound} gives a polynomial bound for every fixed $g$, and with a leading constant also strongly depending on $g$, we currently do not see a reason why the spectral gap should have leading order of the form $n^{Cg}$. Could it be that the chain mixes in polynomial time irrespective of the genus?
\begin{question}
	Are there universal constants $C_1, C_2$ such that $\gamma_{n,g} \geq C_1 n^{C_2}$ for all $n > 2g$?
\end{question}

\subsection*{Notation and layout}
A universal constant $C$ is a constant whose value does not depend on $n$, $g$, or other specific qualities of the objects discussed. A constant $C_g$ depends only on $g$. In what follows, the value of constants such as $C$ or $C_g$ may change from line to line, and even within the same line itself.

We prove irreducibility in \cref{sec:irreducibility}, and the spectral gap lower bound in \cref{sec:spectral_gap_lower_bound}. 

\section{Irreducibility} \label{sec:irreducibility}
In this section we show that apart from one chord diagram of genus $g = n/2$, every two chord diagrams in $\chorddiagrams{n}{g}$ are connected by a sequence of chord swaps.

\begin{definition}
 	Let $n \geq 1$ be an integer. The \emph{Bolza diagram} on the $2n$-sided regular polygon is the chord diagram which matches together opposite sides. 
\end{definition}

This diagram is named after the Bolza surface, a compact hyperbolic surface obtained by gluing together opposite sides of an $8$-sided regular hyperbolic octagon. The Bolza diagram always has genus $g=\floor{n/2}$.

\begin{theorem}\label{thm:irredfixedgenus}
	Let $n\geq 3$. If $g < n/2$, the fixed genus Markov chain on $\chorddiagrams{n}{g}$	is irreducible. If $g=n/2$, then the chain has two classes: one containing just the Bolza diagram as an isolated point, and one containing all other possible diagrams (the case $n=2$ and $g=1$ is trivial and contains only the Bolza diagram).
\end{theorem}

We will prove this theorem by showing that apart from the Bolza diagram for $g=n/2$, every chord diagram has a sequence of chord swaps which ``simplifies'' it, in the following sense.
\begin{definition}
	The length of a chord connecting sides $A$ and $B$ is $1$ plus the number of sides between $A$ and $B$ (the smallest of the two possibilities). A \emph{canonical} chord diagram is one where each chord is either of length $1$, or is of length $2$ and intersects another chord of length $2$. 		
\end{definition}
A canonical diagram of genus $g$ has $g$ pairs of length-$2$ intersecting chords, and $n-2g$ chords of length $1$. In particular, when $n$ is even and $g=n/2$, there are only $4$ canonical diagrams, all rotations of each other; see \cref{fig:canonical_n2_genus}.	
\begin{figure}[H] 
	\centering
	\captionsetup{width=.8\linewidth}
	\includegraphics[width=0.9\textwidth]{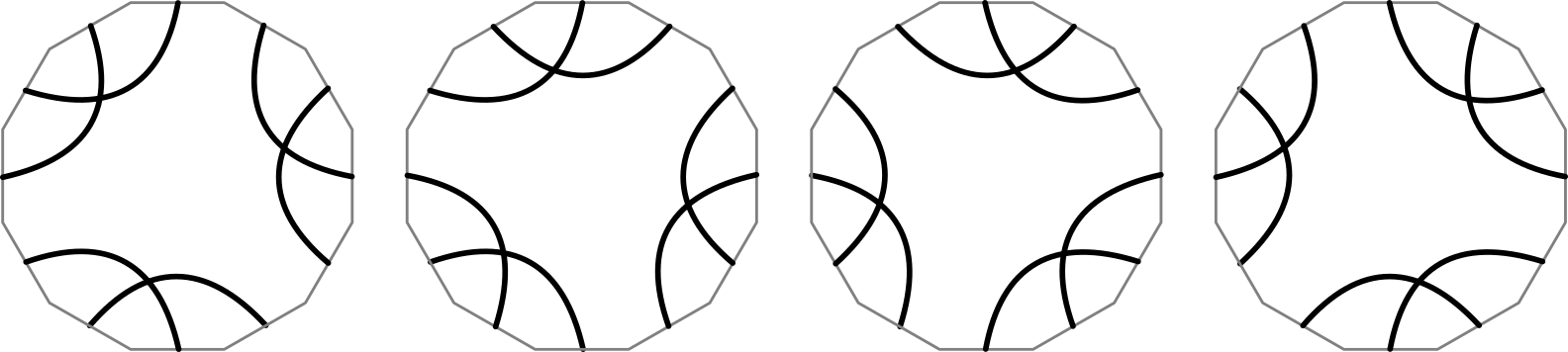}
	\caption{\captionstyle{The four canonical diagrams for genus $g = n/2$.}}
	\label{fig:canonical_n2_genus}
\end{figure}	

In order to understand how a chord swap affects the genus of a diagram, we introduce the following graph.
\begin{definition}
	Let $P$ be a $2n$-gon whose vertices are labelled from $1$ to $2n$ in a clockwise manner, and let $X$ be a chord diagram on $P$. The \textit{gluing graph} of $X$, denoted $G(X)$, is a directed graph whose vertices are the vertices of $P$ and whose edges are
	\begin{equation*}
		E = \bigcup_{\{\{x,x+1\}, \{y,y+1\}\} \in X } \{ (x, y+1), (y, x+1) \} \, ,
	\end{equation*}
	where the vertices $1$ and $2n+1$ are equivalent.
\end{definition}
\begin{figure}[H] 
	\centering
	\captionsetup{width=.8\linewidth}
	\includegraphics[width=0.7\textwidth]{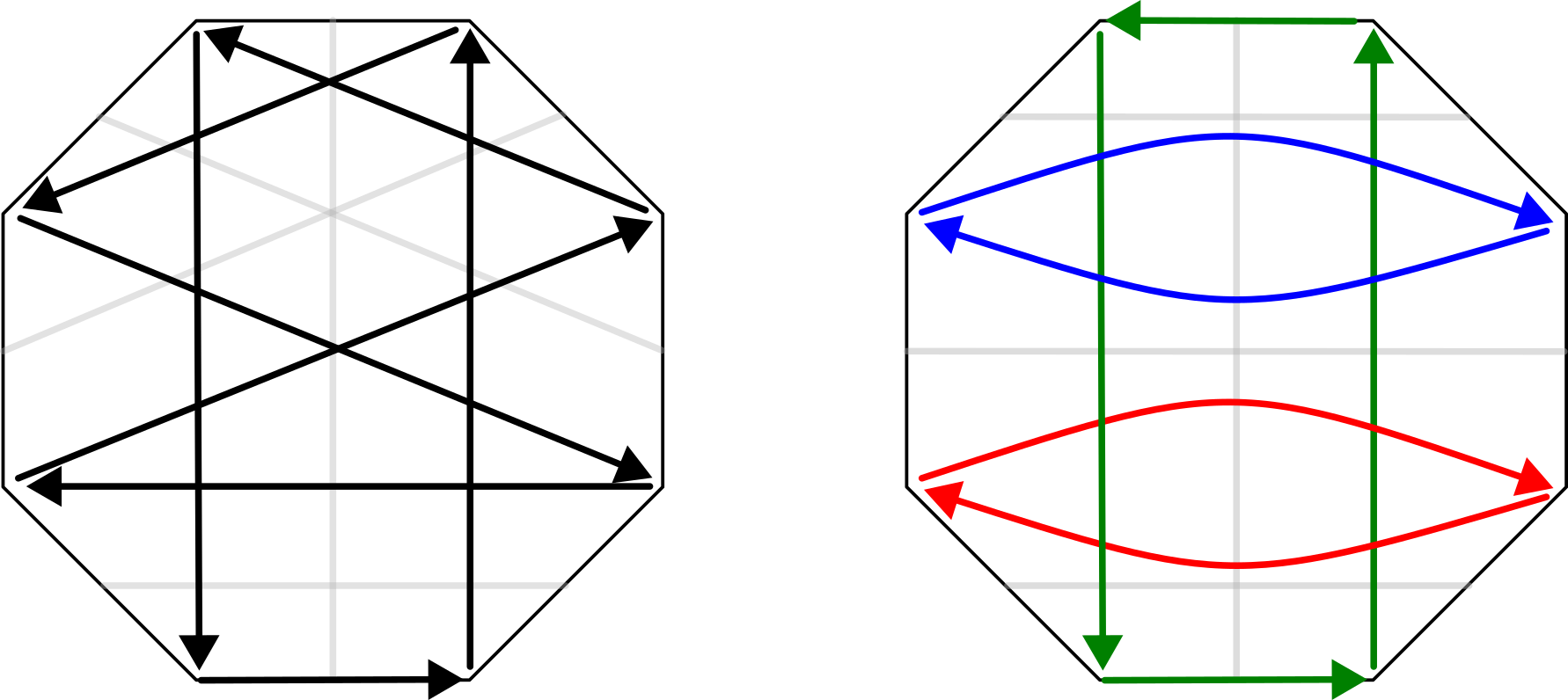}	
	\caption{\captionstyle{Two examples of gluing graphs on the octagon. The underlying chords are shown in grey. The graph on the left has one cycle, and so its diagram is of genus $2$, while the graph on the right has three components, and its diagram has genus $1$.}}
	\label{fig:gluing_graph_examples}
\end{figure}
Every vertex in the gluing graph of $X$ has in-degree and out-degree both equal to one - it is a collection of directed cycles (see \cref{fig:gluing_graph_examples}). The number of components of $G(X)$ is equal to the number of vertices that remain after identifying the edges matched in the chord diagram, and so by \eqref{eq:euler_characteristic}, $X$ has genus $g$ if and only if the number of components in $G(X)$ is $n+1-2g$. Thus, two diagrams $X$ and $Y$ with $n$ chords have the same genus if and only if $G(X)$ and $G(Y)$ have the same number of components. Given this, we can now show the following.
\begin{claim} \label{claim:canonical_can_reach}
	All canonical diagrams with $n$ chords of the same genus can be reached from one another by a finite sequence of genus-preserving chord swaps.
\end{claim}
\begin{proof}	
	As shown in \cref{fig:length_one_shifts}, a length $1$ chord and an adjacent pair of intersecting length $2$ chords may swap positions, so any canonical diagram can reach a diagram where all length $2$ chords are consecutive; all that needs to be shown is that such diagrams can be rotated. If the diagram contains a length $1$ chord, this is done by starting with the length $2$ chord which neighbours a length $1$ chord, and moving its internal side along the ``block'' of length $2$ chords until it reaches the other side. This shifts the length $2$ chords by $1$ side; the remaining chords are non-intersecting and can be rearranged so that all are of length $1$. See \cref{fig:chord_rotations} for an example. If the diagram does not contain length $1$ chords, the same process works by starting from any chord. 
	\begin{figure}[h!] 
		\centering
		\includegraphics[width=0.75\textwidth]{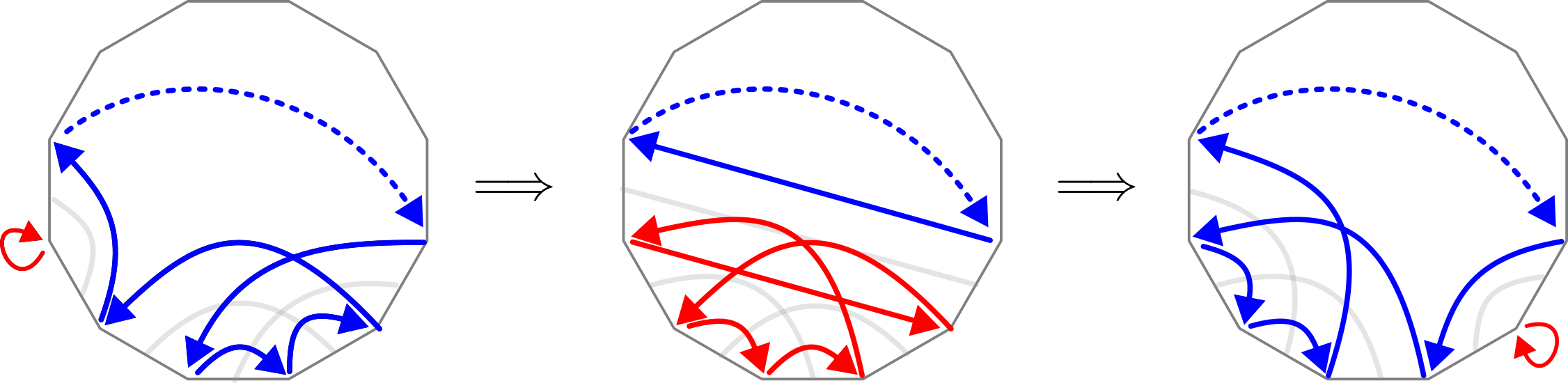}
		\captionsetup{width=.8\linewidth}
		\caption{\captionstyle{It is always possible to swap the position of a length $1$ chord and a pair of intersecting length $2$ chords without changing the genus. The chords are shown in grey, while the gluing graph, which always maintains the same number of components, is shown in colour. The dashed line indicates that the cycle of the gluing graph is connected in some unspecified way.}}
		\label{fig:length_one_shifts}
	\end{figure}	
	\begin{figure}[h!] 
		\centering
		\includegraphics[width=0.9\textwidth]{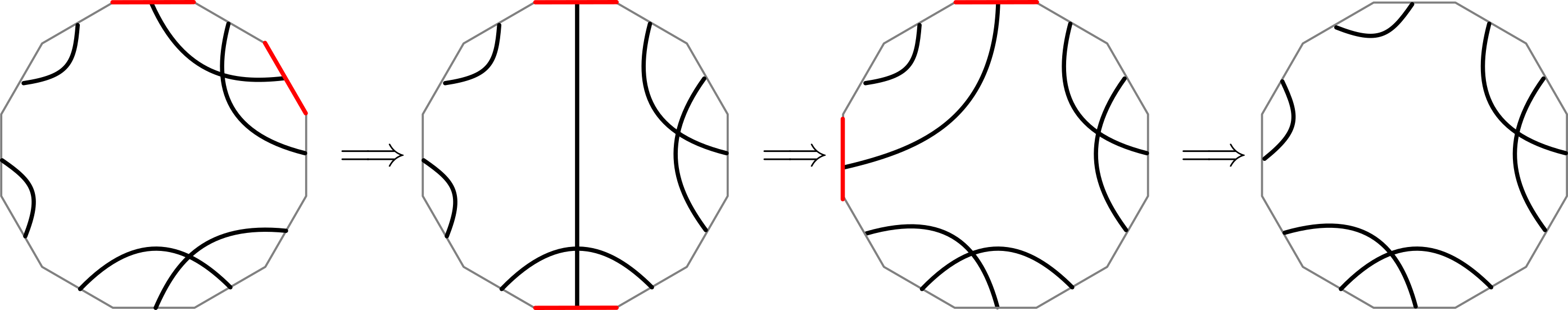}
		\captionsetup{width=.8\linewidth}
		\caption{\captionstyle{It is always possible to rotate a canonical diagram whose length $2$ chords are all consecutive}}
		\label{fig:chord_rotations}
	\end{figure}	
\end{proof}

Given the above claim, to prove \cref{thm:irredfixedgenus}, it suffices to show that any diagram can reach a canonical diagram (apart from the Bolza diagram for $g=n/2$). The analysis depends on the genus.

If $X$ has genus $g=n/2$, then $G(X)$ is just a single directed cycle. In this case, the gluing graph induces a cyclic order on its vertices. For three vertices $a,b,c$, we write $a < b <c $ to mean that when the cycle is traversed starting at $a$, the order of the three vertices seen is $a,b,c$. The vertex that follows $a$ in this cyclic ordering is denoted $a+1$.

The following claim characterizes when a chord swap is possible for genus $g=n/2$.
\begin{claim} \label{claim:when_can_you_swap}
	Let $X \in \chorddiagrams{n}{n/2}$ be a chord diagram, let $\{A,B\}$ and $\{C,D\}$ be two chords, and denote the anticlockwise-most vertices of the sides $A,B,C,D$ by $a,b,c,d$ respectively as in \cref{fig:two_edges_and_labelings} (note that it does not matter if $\{A,B\}$ and $\{C,D\}$ cross or not). The chords $\{A,B\}$ and $\{C,D\}$ may be swapped if and only if the cyclic ordering of the vertices $a,b,c,d,$ in the cycle $G(X)$ is one of the following four:
	\begin{align*}
		&a < b < c < d \\
		&a < b < d < c \\
		&a < c < d < b \\
		&a < d < c < b \, .
	\end{align*}
	\begin{figure}[H] 
		\centering
		\includegraphics[width=0.3\textwidth]{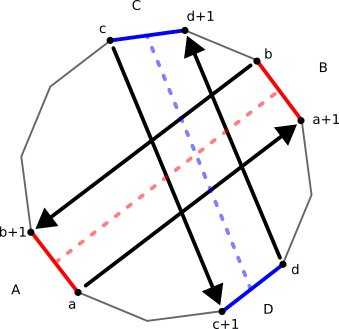}
		\captionsetup{width=.8\linewidth}
		\caption{\captionstyle{The sides $A$ and $B$ are connected by a chord, as well as the sides $C$ and $D$ (the chords are shown in dashed red and blue, respectively). The corresponding directed edges in the gluing graph are drawn in black. We label the anticlockwise-most vertices of the sides $A,B,C,D$ by $a,b,c,d$ respectively. Note that it is possible that some of the vertex labels and their successors coincide, e.g. $a+1=d$. When $g = n/2$, all four vertices are in a single cycle in the gluing graph, and there are six possible cyclic orderings for them.}}
		\label{fig:two_edges_and_labelings}
	\end{figure}	
\end{claim}
\begin{proof}
	\begin{figure}[p] 
		\centering
		\includegraphics[width=\textwidth, height=0.8\textheight, keepaspectratio]{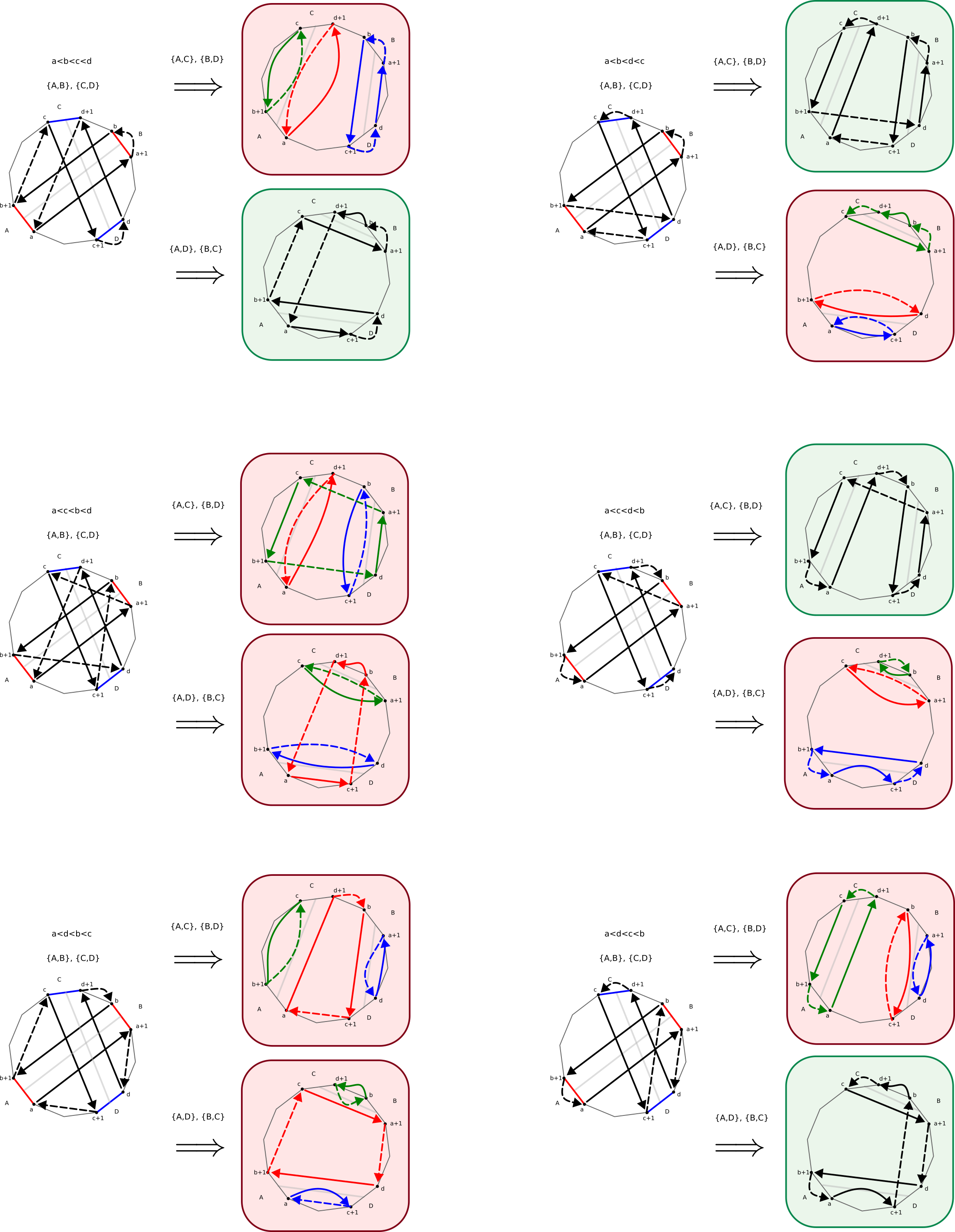}
		\caption{\captionstyle{Whether or not a chord swap preserves the genus depends on the ordering of the endpoints of the matched sides. In each column, the left-hand image shows the gluing graph, based on the ordering written above it. The solid black lines are the edges in the gluing graph due to the sides $A,B,C,D$, while the dashed black lines are unspecified and just show the ordering. The two images to the right are the results of the two possible swaps between the chords $\{A,B\}$  and $\{C,D\}$. The swaps which preserve the cycle structure are highlighted in green, while those which create additional cycles are highlighted in red.}}
		\label{fig:when_can_you_swap}
	\end{figure}
	A chord swap preserves the genus only if it does not change the number of components in $G(X)$. A simple case analysis shows that of the six possible cyclic orderings of the vertices $a,b,c,d$ in $G(X)$, only the orderings $a < c < b < d$ and $a < d < b < c$ do not allow the chords to be swapped; see \cref{fig:when_can_you_swap}. 	
\end{proof}

\begin{claim} \label{claim:can_create_length_2}
	Let $X \in \chorddiagrams{n}{n/2}$. If $X$ is not the Bolza diagram and does not have a chord of length $2$, then there is a genus-preserving chord swap which creates a chord of length $2$. 
\end{claim}
\begin{proof}
	We will show that there are two chords, $\{A,B\}$ and $\{C,D\}$ where the sides $A$ and $C$ are separated by exactly one side, and such that the swap $\{A,B\}, \{C,D\} \to \{A,C\}, \{B,D\}$ preserves the genus. This swap creates the desired chord $\{A,C\}$ of length $2$.
	
	Let $\{A,B\}, \{C,D\} \in X$, and suppose that $A$ and $C$ are separated by another side $E$. Since $g = n/2$, both $B \neq E$ and $D \neq E$; labelling the anticlockwise-most endpoints of the sides $A,B,C,D$ by $a,b,c,d$ respectively, the situation is as in \cref{fig:four_consdecutive}.	
	\begin{figure}[h!] 
		\centering
		\includegraphics[width=60mm]{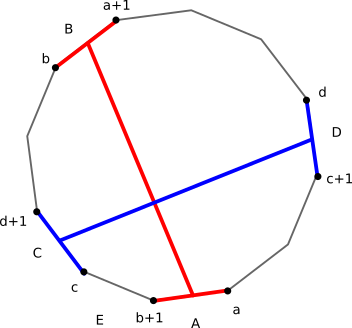}
		\captionsetup{width=.8\linewidth}
		\caption{\captionstyle{Can we find sides $A$ and $C$, which are separated by exactly one side $E$, such that $a < c < d < b$?}}
		\label{fig:four_consdecutive}
	\end{figure}	
	By \cref{claim:when_can_you_swap}, in order to make the swap $\{A,B\}, \{C,D\} \to \{A,C\}, \{B,D\}$, it suffices that the endpoints $a,b,c,d$ satisfy $a<c<d<b$ in the ordering induced by the gluing graph $G(X)$. The relation $a<c<d<b$ is equivalent to $a < c < d+1 < b+1$. Since $a,b+1,c,d+1$ are clockwise consecutive vertices on the polygon, our goal is reduced to finding four clockwise consecutive vertices $x,y,z,w$ on the polygon such that $x < z < w < y$ (we can then substitute $a=x$, $b+1=y$, $z=c$ and $d+1=w$). We now show that this is always possible. 	
	
	For a vertex $v$, let $\Delta(v)$ be the number of steps in the gluing graph required to get from $v$ to the vertex directly clockwise of $v$ on the polygon $P$. This is always bounded above by $2n-1$. Let $\{A,B\} \in X$ be two sides, and denote their right endpoints by $a,b$ respectively. Then $\Delta(a) + \Delta(b) = 2n+2$; indeed, in the gluing graph, getting from $a$ to the vertex clockwise from it requires first getting from $a$ to $b$ and then taking one more step, while getting from $b$ to the vertex clockwise from it requires first getting from $b$ to $a$ and then taking one more step. Since the cycle has length $2n$, the total number of steps is $2n+2$. Summing over all chords in the matching, we thus have
	\begin{equation} \label{eq:summing_delta}
		\sum_{v \in P} \Delta(v) = n(2n+2) \, .
	\end{equation}
	We claim that when $X$ is not the Bolza diagram, then there are three clockwise consecutive vertices $x,y,z$, such that $\Delta(x) + \Delta(y) \geq 2n+2$, but $\Delta(y) + \Delta(z) < 2n +2$. Arguing by contradiction, if there are no such clockwise consecutive triplets $x,y,z$, then whenever $\Delta(x) + \Delta(y) \geq 2n+2$, we must also have $\Delta(y) + \Delta(z) \geq 2n+2$. By \eqref{eq:summing_delta}, summing over clockwise consecutive vertices gives 
	\begin{equation} \label{eq:summing_pairs_of_differences}
		\sum_{x,y \in P, y\text{ clockwise to }x} \Delta(x) + \Delta(y) = 2n(2n+2) \, .
	\end{equation}
	Since there are $2n$ vertices, by the pigeonhole principle there exists a pair of clockwise consecutive vertices $x,y$ such that $\Delta(x) + \Delta(y) \geq 2n + 2$. By the contrapositive assumption, this implies that $\Delta(u) + \Delta(v) \geq 2n+2$ for \textbf{every} pair of clockwise consecutive vertices $u,v$; but by \eqref{eq:summing_pairs_of_differences} we cannot have any strict inequality, and in fact $\Delta(u) + \Delta(v) = 2n+2$ for every such pair. In particular, as we traverse clockwise along the vertices of the polygon, there is an integer $k$ such that the value of $\Delta(v)$ alternates between $k$ and $2n+2-k$. 
	
	Now, let $\{A,B\} \in X$ be two sides whose anticlockwise-most endpoints are $a$ and $b$, respectively. Let $a = x_1, x_2, \ldots, x_\ell  = b$ be the vertices on the polygon going from $a$ to $b$ in clockwise order. Consider the path in the gluing graphs that visits $x_1, \ldots, x_\ell$ in order, and then visits the vertex clockwise to $b$; this path has length $\sum_{i=1}^{\ell} \Delta(x_i)$. Since the vertex clockwise to $b$ is equal to $a+1$, and since any path from $a$ to $a+1$ in the gluing graph must consist of an integer number of cycles plus the one edge between $a$ and $a+1$, we have 
	\begin{equation*}
		\sum_{i=1}^{\ell} \Delta(x_i) = 1 \mod 2n \, .
	\end{equation*}
	Since, as seen above, every two clockwise consecutive vertices have $\Delta(u) + \Delta(v) = 2 \mod 2n$, $\ell$ must be odd, and by grouping together pairs of clockwise consecutive vertices starting from $x_2$, we have
	$$\frac{\ell-1}{2}\cdot 2 + \Delta(a) = 1 \mod 2n \, . $$
	Reversing the roles of $A$ and $B$, denoting by $m$ the clockwise distance from $b$ to $a$, and grouping together pairs of clockwise consecutive vertices starting from $b$ similarly gives 
	$$\frac{m-1}{2}\cdot 2 + \Delta(a) = 1 \mod 2n \, . $$
	This implies that the clockwise distance from $a$ to $b$ is the same as the clockwise distance from $b$ to $a$, which holds only when $A$ and $B$ are opposite sides of the polygon. As this holds true for every pair of matched sides, $X$ must be the Bolza diagram, and we arrive at a contradiction. There must therefore indeed exist three clockwise consecutive vertices $x,y,z$ with both $\Delta(x) + \Delta(y) \geq 2n+2$ and $\Delta(y) + \Delta(z) < 2n +2$.
	
	Letting $w$ be the vertex clockwise to $z$, we now show that $x < z < w < y$ under the gluing graph cyclic order. Without loss of generality, we may assign numbers $n_x, n_y, n_z, n_w$ indicating their position in the gluing graph in order to mark their ordering:
	\begin{align*}
		&n_x = 0 \\
		&n_y = \Delta(x) \\
		&n_z = \Delta(x) + \Delta(y) && \mod 2n \\
		&n_w = \Delta(x) + \Delta(y) + \Delta(z)&&\mod 2n \, .
	\end{align*}
	Having chosen $n_x=0$, showing that $x < z < w < y$ in the cyclic ordering of the gluing graph is the same as showing that $n_x < n_z < n_w < n_y$. Since $\Delta(x) + \Delta(y) \geq 2n$ and $\Delta(y) \leq 2n-1$, we in fact have $n_z = \Delta(x) + \Delta(y) - 2n \leq \Delta(x) + 2n -1 -2n < \Delta(x) = n_y$. As for $n_w$, since $\Delta(y) + \Delta(z) \leq 2n+1$, we have $n_w \leq \Delta(x) + 2n + 1 - 2n = n_y + 1$. But $n_w$ cannot equal $n_y+1$ - if it were so, then the side $\{y,z\}$ would be matched to the side $\{z,w\}$, and this is impossible in a genus $n/2$ chord diagram. Since it cannot also equal $n_y$, we must have $n_w < n_y$ as needed. Finally, observe that $n_w > n_z$, since $n_w$ is obtained from $n_z$ by adding a positive number but without an additional modulo operation.
\end{proof}

\begin{claim} \label{claim:can_cross_length_2}
	Let $X \in \chorddiagrams{n}{n/2}$. If $X$ is not the Bolza diagram and does not have a pair of chords of length $2$ which cross each other, then there is a finite sequence of genus-preserving chord swaps which creates such a pair.
\end{claim}
\begin{proof}
	By \cref{claim:can_create_length_2}, we can assume that $X$ has a chord of length $2$. Denoting the chord's matched sides by $A$ and $B$, there must exist $\{C,D\} \in X$ such that $C$ is between $A$ and $B$; the chord $\{C,D\}$ has length strictly greater than $2$. Consider the side $E$ which is clockwise of $D$, and its paired-up side $F$ (see top of \cref{fig:decreasing_length_of_edge_under_arc}). Then the swap $\{C,D\},\{E,F\} \to \{C,E\},\{D,F\}$ preserves the genus. Indeed, the gluing graph $G(X)$ induces two possible orderings: either $x + 4 < y< z$, or $x + 4 < z < y$. In both cases, as is seen in \cref{fig:decreasing_length_of_edge_under_arc}, the gluing graph has a single cycle after the swap. This swap shortens the length of the chord with endpoint $C$; repeating this procedure up to $n$ times, we obtain two chords of length $2$ which cross each other, as needed.	
	\begin{figure}[H] 
		\centering
		\includegraphics[width=0.6\textwidth]{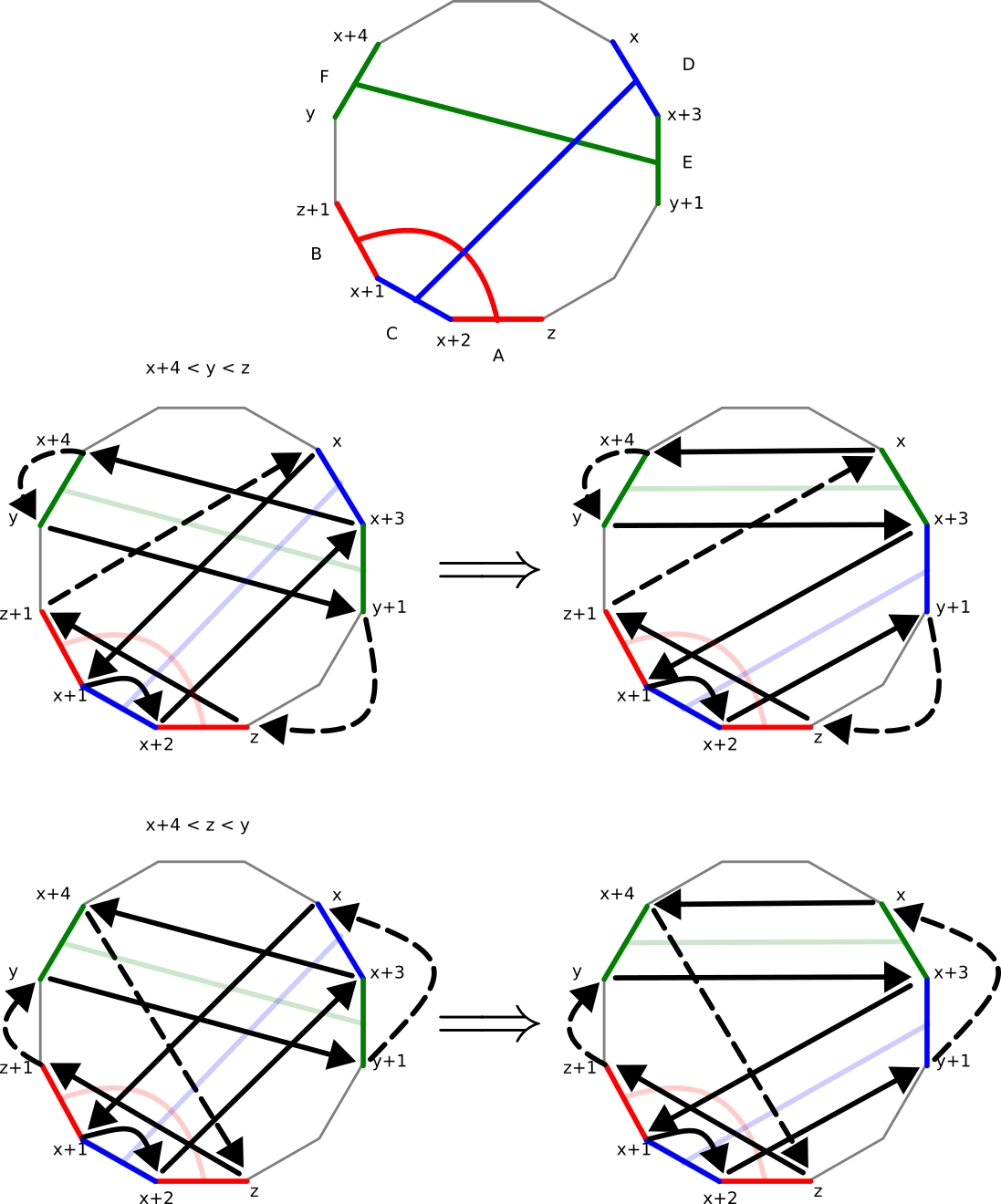}
		\captionsetup{width=.8\linewidth}		
		\caption{\captionstyle{The chord $\{C,D\}$ has length greater than $2$, and one of its endpoints is under a chord of length $2$. Its other endpoint can always be shifted closer to this chord by a swap, eventually yielding two chords of length $2$ which cross each other. The solid black lines are edges of the gluing graph, while the dashed black lines are unspecified and just show the ordering.}}
		\label{fig:decreasing_length_of_edge_under_arc}
	\end{figure}
\end{proof}

\begin{claim}\label{claim:can_create_length_1}
	Let $g < n/2$, and let $X \in \chorddiagrams{n}{g}$. If $X$ does not have a chord of length $1$, then there exists a sequence of genus-preserving chord swaps which creates one. 
\end{claim}

\begin{proof}
	Let $m$ be the number of connected components of the gluing graph $G(X)$, and label the vertices of the polygon by $1,\ldots,m$ according to the component to which they belong. Since $g<n/2$, we must have $m \geq 2$. 
	
	If the polygon $P$ has three clockwise consecutive vertices such that either
	\begin{enumerate}
		\item all three have different labels; OR
		\item the first two vertices have the same label, and the third has a different label
	\end{enumerate} 
	then there is a single genus-preserving chord swap which creates a chord of length $1$; see \cref{fig:two_and_three_in_a_row}. When $m \geq 3$, there always exists three clockwise consecutive vertices satisfying one of these conditions, so it remains to consider the case where $m = 2$ and the vertices of the polygon are alternately labelled. 
	\begin{figure}[h!] 
		\centering
		\includegraphics[width=0.70\textwidth]{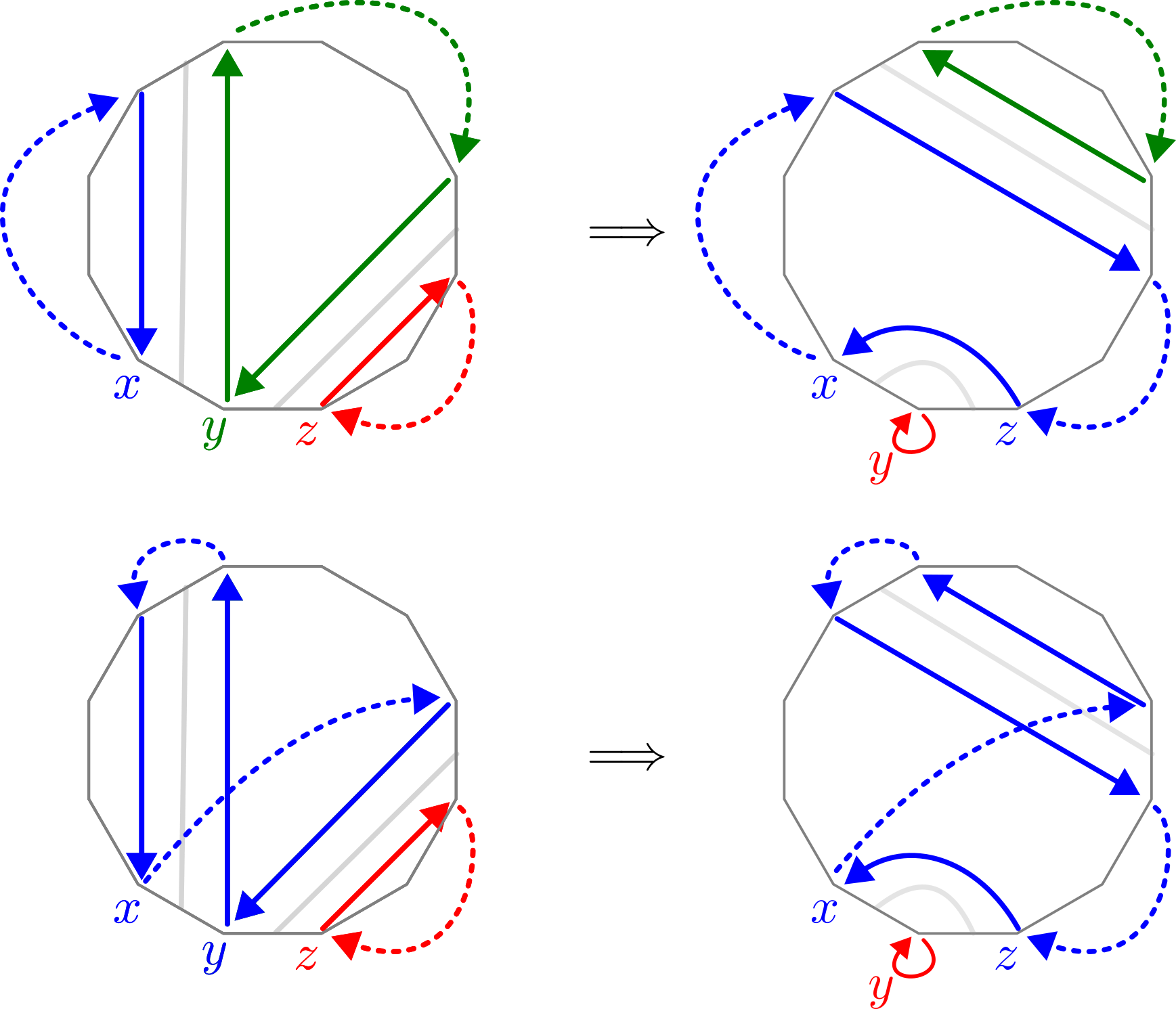}
		\captionsetup{width=.8\linewidth}
		\caption{\captionstyle{When there are three vertices with different labels (top) or two vertices of one label and a third vertex of another (bottom), it is always possible to create a length $1$ chord with a single genus-preserving swap. The chords are shown in grey, while the gluing graph is shown in colour. Solid lines are edges of the gluing graph, while dashed lines are unspecified only indicate the gluing graph's vertex connectivity.}}
		\label{fig:two_and_three_in_a_row}
	\end{figure}		
	Let $e = \{A,B\}$ be the shortest chord in $X$. This chord separates the sides of $P \setminus \{A,B\}$ into two sets; let $S$ be the smaller of the two. Any chord with one endpoint in $S$ must have its other endpoint in $P \setminus S$; otherwise, it would be shorter than $e$. Let $f$ be the chord whose endpoint in $S$ is adjacent to $A$. Then $e$ and $f$ can be swapped while preserving the genus to create a new chord diagram $X'$ with new chords, $e'$ and $f'$; see \cref{fig:alternating_colours}. In a diagram where the labels alternate, every chord must have odd length. In particular, $e$ has odd length; but the chord $e'$ has length ones less than $e$, and so has even length, and so $X'$ vertices are no longer alternately labelled. The argument in the previous paragraph can be applied to $X'$.
		\begin{figure}[h!] 
		\centering
		\includegraphics[width=0.70\textwidth]{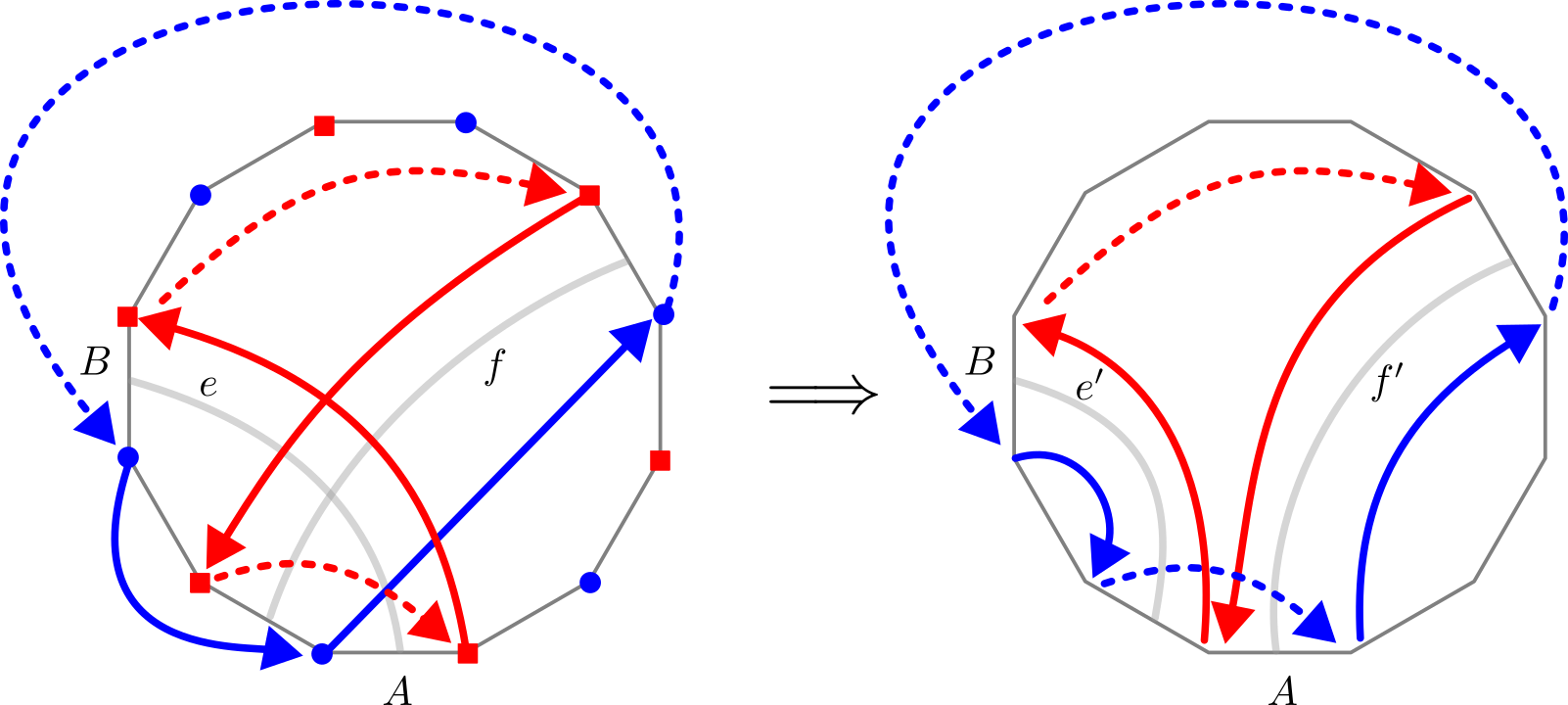}
		\captionsetup{width=.8\linewidth}
		\caption{\captionstyle{When the vertices alternate between two labels, the shortest chord (which has odd length) can always be shortened by $1$, creating a diagram with a non-alternating labelling. The chords $e$ and $f$ are shown in grey, while the gluing graph is shown in colour. Solid lines are edges of the gluing graph, while dashed lines are unspecified and only indicate the gluing graph's vertex connectivity. The connectivity of the red dashed edges on the left is forced, otherwise there would be three components.}}
		\label{fig:alternating_colours}
	\end{figure}	
\end{proof}

\begin{proof}[Proof of \cref{thm:irredfixedgenus}]	 	
	We will show that apart from the Bolza diagram for the case of $g = n/2$, any $X \in \chorddiagrams{n}{g}$ can reach a canonical diagram. Together with \cref{claim:canonical_can_reach}, this will show that the Markov chain graph is connected.
	
	\paragraph{The case $g = n/2$}		
	For the Bolza diagram, if $a,b$ are the anticlockwise-most vertices of the two sides $A,B$ respectively, then $a$ and $b$ are separated by $n$ vertices in the cycle ordering (see \cref{fig:bolza_polygon_ordering}). Thus, for any two pairs of matched sides $\{A,B\}, \{C,D\}$, either $a < c < b < d$ or $a < d < b < c$ in the cycle ordering, and so by \cref{claim:when_can_you_swap}, no move can be made: the Bolza diagram is isolated.
	\begin{figure}[h] 
		\centering
		\includegraphics[width=0.5\textwidth]{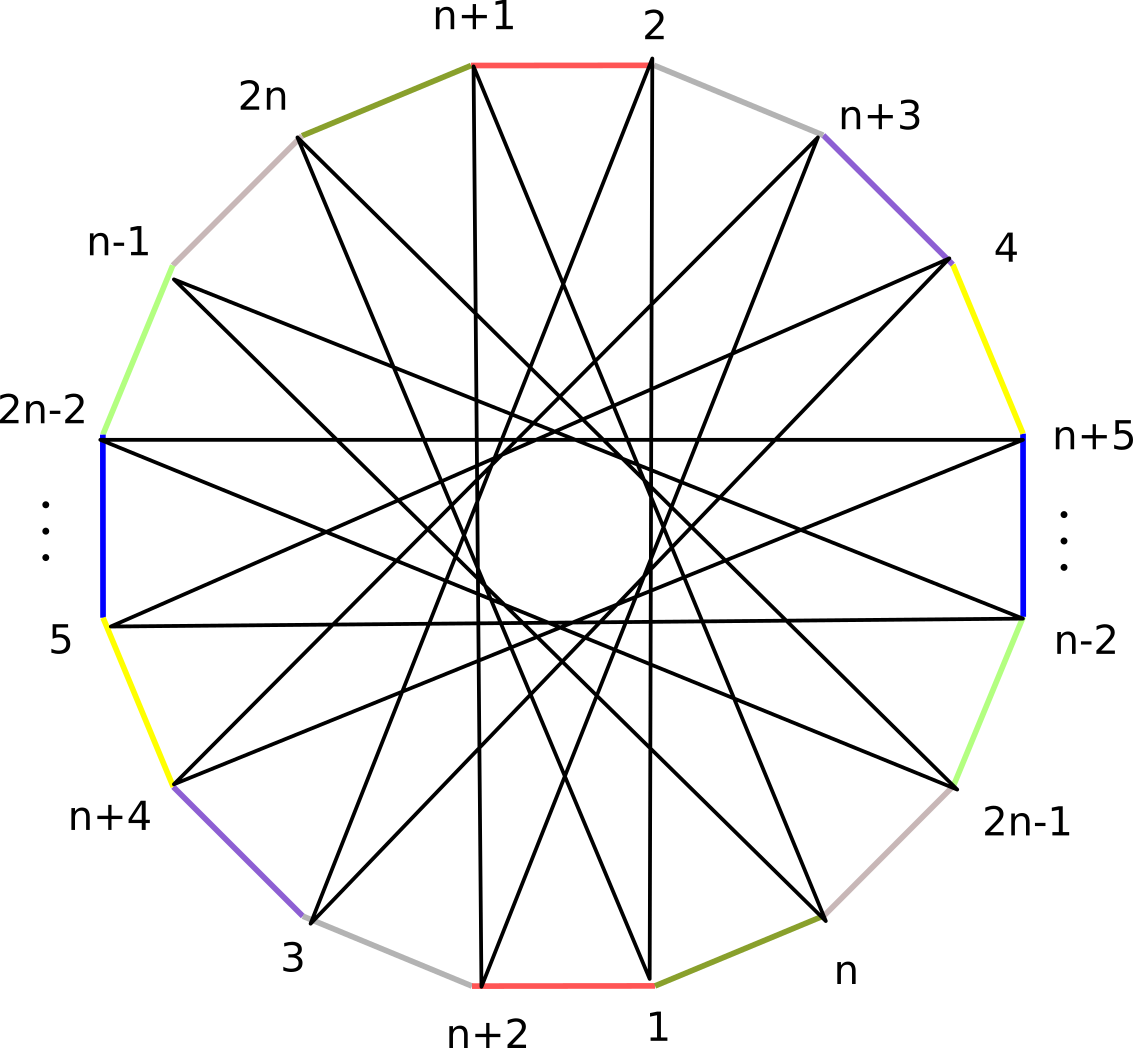}
		\captionsetup{width=.8\linewidth}
		\caption{\captionstyle{In the Bolza diagram, the anticlockwise-most vertices of opposing sides are separated by $n$ other vertices.}}
		\label{fig:bolza_polygon_ordering}
	\end{figure}
		
	We now assume that $X$ is not the Bolza diagram, and show that it can reach a canonical diagram. We assume by induction that it is true for all diagrams with $k$ chords for $k=2,4,\ldots, n-2$ (for $n=2$, there is only a single genus $1$ diagram, which is canonical).
		
	By \cref{claim:can_cross_length_2}, we can assume that $X$ has two chords $\{A,B\}$ and $\{C,D\}$ of length $2$ which cross each other, where $A,B,C,D$ are consecutive polygon sides. Consider the chord diagram $X' \in \chorddiagrams{n-2}{n/2-1}$ on the polygon $P'$ obtained by contracting to a point the sides $A,C,B,D$ in $P$ and restricting the chords of $X$ to the remaining sides. Any genus-preserving move on $X'$ naturally corresponds to a genus-preserving move in $X$.
		
	If $X'$ is not the Bolza diagram on $P'$, then by induction it can reach the canonical diagram which has two pairs of crossing length-$2$ chords, with one pair on each side of where the polygon sides $A,C,B,D$ were before they were contracted. Since $A,C,B,D$ is also a pair of crossing length-$2$ chords, we are done.
	
	Otherwise, $X'$ is the Bolza diagram on $P'$. When $n=4$, $X'$ is a diagram on a polygon with $4$ sides and is already a canonical diagram, so we assume $n > 4$. In this case, even though the Bolza diagram itself is isolated, it is possible to use the sides $A,C,B,D$ as a ``reservoir'' in order to ``unravel'' $X$ into a canonical diagram. Let $E$ be the side at distance $2$ anticlockwise to $A$, and $F$ be the side matched to it in $X$. A simple analysis of the ordering induced by the gluing graph gives that $e < a < b < f$ (see the left part of \cref{fig:subpolygon_bolza}), and so by \cref{claim:when_can_you_swap}, we can make the swap $\{A,B\}, \{E,F\} \to \{A,E\},\{B,F\}$ (see the middle part of \cref{fig:subpolygon_bolza}). In the new diagram, the chord $\{A,E\}$ has length $2$, and there is a chord of length greater than $2$ with an endpoint between $A$ and $E$. As in the proof of \cref{claim:can_cross_length_2}, there is a sequence of moves which reduces the length of this longer chord until it has length $2$ as well; this sequence of moves does not alter the chord $\{C,D\}$. After doing so, the diagram will have two crossing chords of length $2$, while the rest of the diagram will be different than the Bolza diagram on $2(n-2)$ sides (see the right part of \cref{fig:subpolygon_bolza}). We can now use the inductive argument of the preceding paragraph.
	\begin{figure}[H] 
		\centering
		\includegraphics[width=120mm]{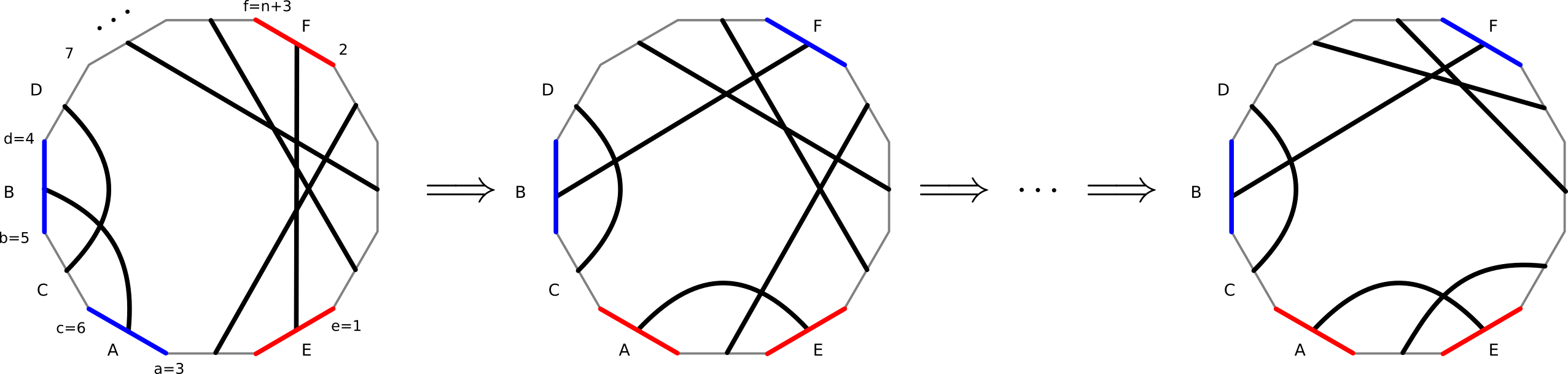}
		\captionsetup{width=.8\linewidth}
		\caption{\captionstyle{When there are two crossing chords of length $2$ together with a Bolza diagram, it is possible to create two different crossing chords so that the rest of the matching is not a Bolza diagram.}}
		\label{fig:subpolygon_bolza}
	\end{figure}				
	
	\paragraph{The case $g < n/2$} 			
	We now show that any chord diagram $X \in \chorddiagrams{n}{g}$ with $g < n/2$ can reach a canonical diagram.
	
	For $n=2$ this is trivially true as there are only two chord diagrams with genus $g=0$, both canonical. We assume by induction that it is true for all diagrams with $k$ edges, $k=2,\ldots,n-1$ and genus $< k/2$. By \cref{claim:can_create_length_1}, we can assume that $X$ has a chord $e$ of length $1$. Consider the chord diagram $X' \in \chorddiagrams{n-1}{g}$ on the polygon $P'$ obtained by contracting to a point the sides of $e$ and restricting the chords of $X$ to the remaining sides. 
	
	If $g < (n-1)/2$, by induction $X'$ can reach a canonical diagram, and we are done. If $g = (n-1)/2$ and $X'$ is not the Bolza diagram, then $X'$ can also reach a canonical diagram by the first part of this proof. Finally, if $X'$ is the Bolza diagram on $P'$, then it is possible to make two genus-preserving chord swaps on $X$ and create a new diagram with a pair of intersecting chords of length $2$: first extend the length $1$ chord to a length $2$ chord, and then use the chord whose endpoint is in between the endpoints of the length $2$ chord to create another length $2$ chord (see \cref{fig:one_plus_bolza}). When contracting the sides of the two length $2$ chords, we obtain a diagram $X''$ which, by induction, can reach a canonical diagram.	
	\begin{figure}[H] 
		\centering
		\includegraphics[width=0.75\textwidth]{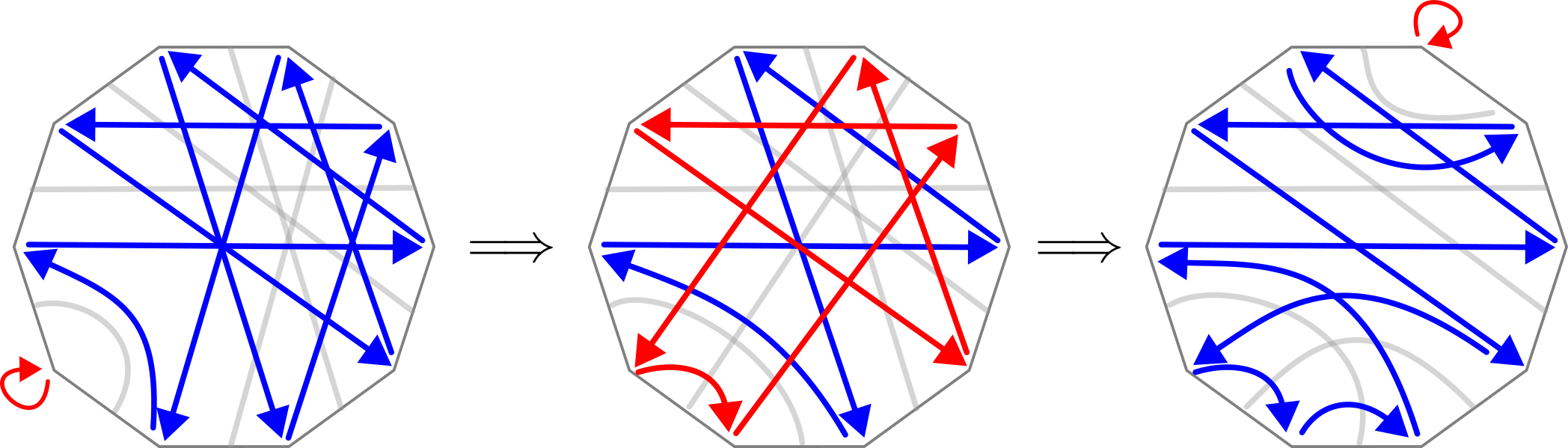}
		\captionsetup{width=.8\linewidth}
		\caption{\captionstyle{A diagram consisting of a length $1$ chord and a Bolza diagram can create two length $2$ intersecting chords. The chords are shown in grey, while the gluing graph is shown in colour. }}
		\label{fig:one_plus_bolza}
	\end{figure}		
\end{proof}

The proof of \cref{thm:irredfixedgenus} shows that the diameter of the connected component of the Markov graph on $\chorddiagrams{n}{g}$ is asymptotically bounded above by $C(n + g^2)$ for some constant $C > 0$: it takes a constant number of moves to create $n-g$ chords of length $1$, leaving an induced diagram of maximum genus $g$. For this induced diagram, it takes up to $g$ swaps to create a pair of intersecting length $2$ chords, and in the worst case this may need to happen $g$ times in order to get to a canonical diagram. We conjecture that this is the correct asymptotic:
\begin{conjecture}
	The diameter of the connected component of the Markov graph on $\chorddiagrams{n}{g}$ is of order $n+g^2$.
\end{conjecture}
For $g=0$, the diameter of $\chorddiagrams{n}{0}$ is exactly $n-1$. 

\section{Spectral gap lower bound} \label{sec:spectral_gap_lower_bound}
We will consider the Markov chain on $\chorddiagrams{n}{g}$ as a combination of two different walks: one is a product of walks on smaller-sized chord diagrams of genus $0$, while the other is a walk on what is essentially the ``simplest'' genus $g$ chord diagram, which has only a constant number of chords. Consider for example \cref{fig:small_genus_simplification}, whose left part shows a genus $1$ chord diagram on $n=100$ vertices. Although this diagram appears complicated, many of the chords are in some sense ``redundant''; all of the topology of the chord diagram stems from a simple structure, and the rest of the chords do not contribute to it. If we ignore these redundant chords, our state would be essentially equivalent to the diagram shown in the right part of \cref{fig:small_genus_simplification}, which is one of the $3$ topologies of chord diagrams of genus $1$ with $n=3$.

\begin{figure}[h!] 
	\centering
	\includegraphics[width=0.5\textwidth]{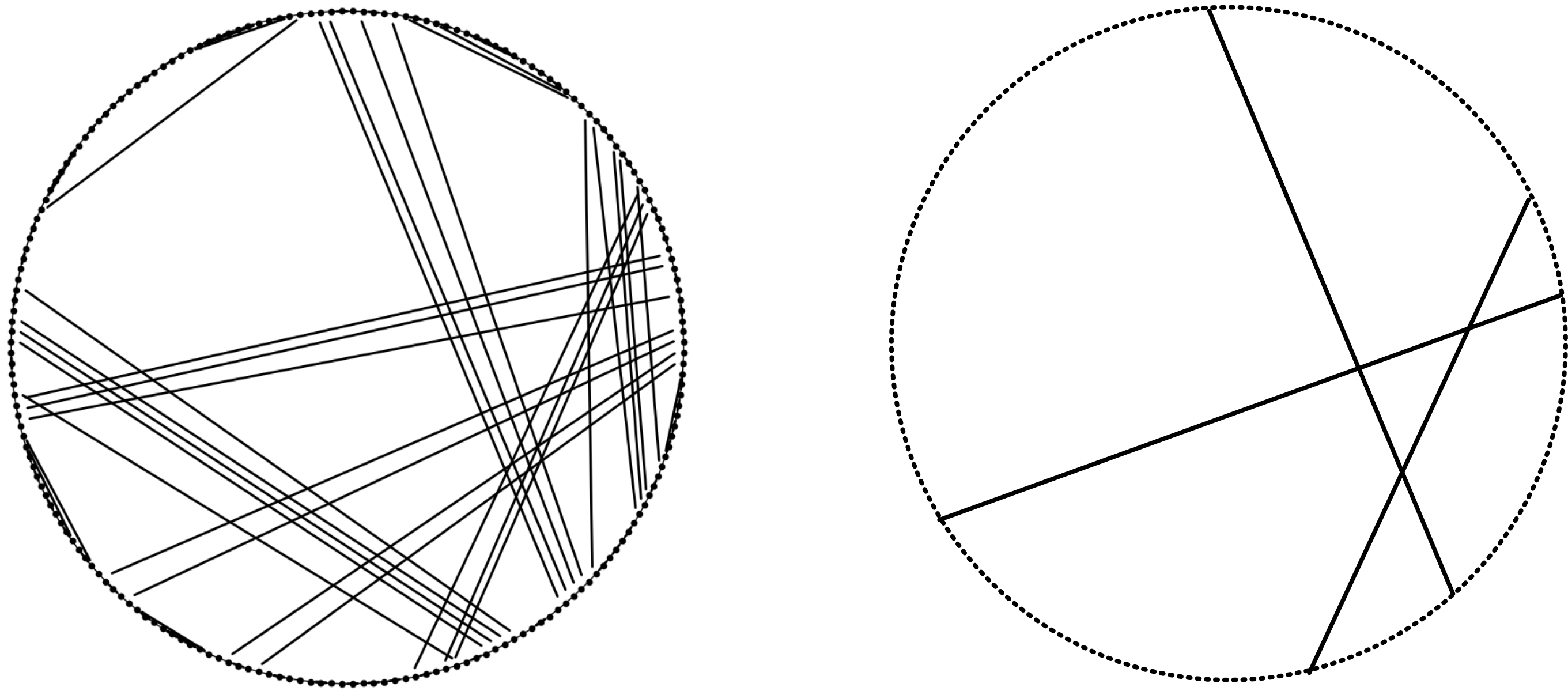}
	\captionsetup{width=.8\linewidth}
	\caption{\captionstyle{\textbf{Left}: A chord diagram of genus $1$ on a $200$ sided polygon. \textbf{Right}: Essentially, all of the topology of the diagram is encapsulated by a small number of chords.}}
	\label{fig:small_genus_simplification}
\end{figure}

More formally, we use the restriction-projection framework (originally developed by Madras and Randall \cite{madras_randall_markov_chain_decomposition}) using the notation of \cite{jerrum_et_all_decomposable_markov_chains}, which goes as follows. Let $P(x,y)$ be a time-reversible irreducible Markov chain on a state space $\Omega$ with stationary distribution $\pi$. Partition $\Omega$ into $m$ sets $\Omega_1, \ldots, \Omega_m$, and define the distribution $\overline{\pi}:[m] \to [0,1]$ by 
\begin{equation*}
	\overline{\pi}(i) = \sum_{x \in \Omega_i} \pi(x) \, .
\end{equation*}
The projection chain $\overline{P} : \{\Omega_1, \ldots, \Omega_m\} \times \{\Omega_1, \ldots, \Omega_m\} \to [0,1]$ is defined by 
\begin{equation*}
	\overline{P}(\Omega_i,\Omega_j) = \frac{1}{\overline{\pi}(i)} \sum_{\substack{x \in \Omega_i \\ y \in \Omega_j}} \pi(x) P(x,y) \, .
\end{equation*}
This chain is irreducible and has $\overline{\pi}$ as a stationary distribution.

For $i \in [m]$, the restriction chain $P_i : \Omega_i \times \Omega_i \to [0,1]$ is defined as 
\begin{equation*}
P_i(x,y) = \begin{cases}
	P(x,y) & x \neq y \\
	1- \sum_{z \in \Omega_i \setminus \{x\} }P(x,z) & x = y \, .
\end{cases}
\end{equation*}
We will assume that this chain is also irreducible, in which case it has stationary distribution $\pi(x)/\overline{\pi}(i)$. 

\begin{theorem}[{\cite[Theorem 1]{jerrum_et_all_decomposable_markov_chains}}]		
	\label{thm:projection_restriction_poincare}
	For $i \in [m]$, let $P_i$ satisfy the Poincar\'{e} inequality with constant $\gamma_i$, and denote $\gamma_{\min} = \min \{\gamma_1, \ldots, \gamma_m\}$. Let $\overline{P}$ satisfy the Poincar\'{e} inequality with constant $\overline{\gamma}$. Then the original Markov chain $P$ satisfies the Poincar\'{e} inequality with 
	\[
		\gamma = \overline{\gamma}\gamma_{\min} / 4 \, .
	\]
\end{theorem}

We will apply \cref{thm:projection_restriction_poincare} to a suitable decomposition of the original chain $P$ into restriction and projection chains. This decomposition identifies which chords are important to the topology of the chord diagram.

\pagebreak
Let $g \geq 1$, let $X \in \chorddiagrams{n}{g}$, and let $C_{\mathrm{cross}}$ be the set of chords of $X$ with at least one crossing. 
\begin{definition} \label{def:flatly_parallel}
	Let $e_{1} = \{A,B\},e_{2} = \{C,D\} \in C_{\mathrm{cross}}$ be two chords which do not cross each other. The chords partition the sides of $P \setminus \{A,B,C,D\}$ into three disjoint sets: 
	\begin{itemize}
		\item The set $I^{e_1}$ of sides which are separated from $e_2$ by $e_1$.
		\item The set $I^{e_2}$ of sides which are separated from $e_1$ by $e_2$.
		\item The set $I^{e_1,e_2}$ of sides which are between $e_1$ and $e_2$.
	\end{itemize}
	The chords $e_{1}$ and $e_{2}$ are said to be \textit{flatly parallel} if every chord with an endpoint in $I^{e_1,e_2}$ has both endpoints in $I^{e_1,e_2}$, and none of those chords intersect each other.
\end{definition}

\begin{figure}[h!] 
	\centering
	\includegraphics[width=0.3\textwidth]{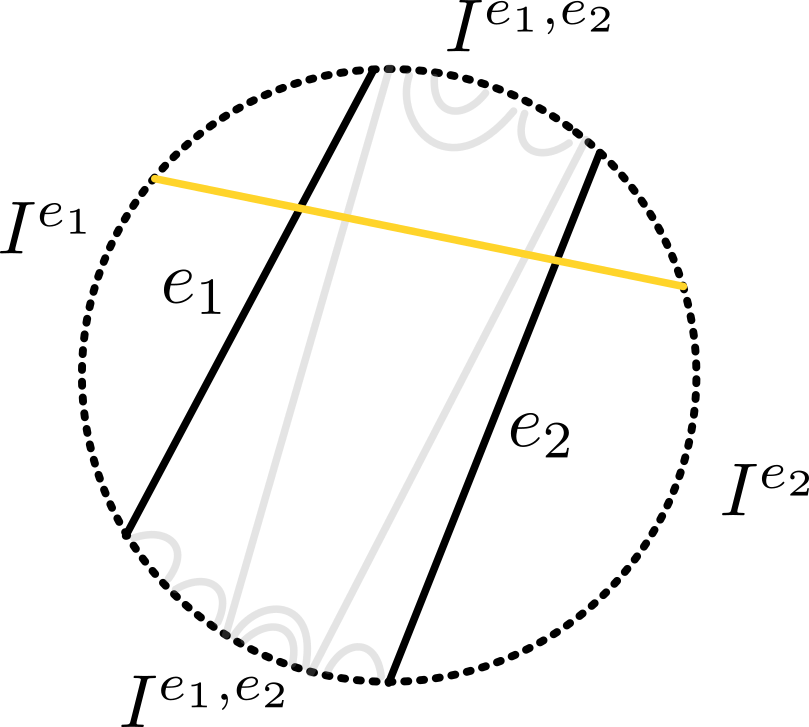}
	\captionsetup{width=.8\linewidth}
	\caption{\captionstyle{An example of flatly chords $e_1$ and $e_2$, shown in black. The chords in $I^{e_1,e_2}$ are shown in grey, while a chord which crosses $e_1$ and $e_2$ is shown in yellow (there must be at least one such chord).}}
	\label{fig:flatly_parallel}
\end{figure}

\begin{claim} \label{claim:crossing_parallel_chords}	
	Let $e_1$ and $e_2$ be flatly parallel. Then any chord $f$ which intersects $e_1$ also intersects $e_2$.		
\end{claim}

\begin{proof}
	By definition, $f$ cannot have both endpoints in $I^{e_1}$, $I^{e_2}$, or $I^{e_1, e_2}$, since then it would not intersect $e_1$. It cannot have a single endpoint in $I^{e_1, e_2}$, since then it would have both endpoints there. It therefore has one endpoint in $I^{e_1}$ and one in $I^{e_2}$, and so intersects $e_2$.		
\end{proof}

Being flatly parallel is an equivalence relation. Thus the set $C_{\mathrm{cross}}$ is partitioned into $N(X)$ flatly parallel sets $C_{\mathrm{cross}}=\bigcup_{i=1}^{N(X)}F_{i}$.
\begin{lemma}\label{lem:number_of_flatly_parallel}		
	$N(X) \leq 6g-3$.
\end{lemma}

\begin{proof}
	For this proof, it is convenient to work with unicellular maps, an object equivalent to chord diagrams. For self-completeness, we briefly describe the connection between the two. A combinatorial map is a connected multigraph, where each vertex has a cyclic ordering of the endpoints of its edges (with self-loops having two endpoints). Its faces are the closed cycles obtained by walking along the graph, where at every step the next edge is given by the cyclic order of edges of the current vertex. A map is called \textit{unicellular} if it has only one face. The genus of a map is given by the Euler relation $V + F - E = 2 - 2g$. A map is called rooted if it has a distinguished oriented edge.
	
	Let $\surfacemaps{n}{g}$ be the set of all rooted unicellular combinatorial maps of genus $g$ with $n$ edges. There is a bijection between $\chorddiagrams{n}{g}$ and $\surfacemaps{n}{g}$:
	\begin{itemize}
		\item Given a unicellular map $(M,e) \in \surfacemaps{n}{g}$, consider the cycle corresponding to the single face of $M$ which has $e$ as its first edge. This walk has length $2n$, and visits every edge in $M$ exactly twice. The corresponding chord diagram has a chord between the sides $i$ and $j$ of the polygon $P$ whenever the same edge of $M$ appears at positions $i$ and $j$ in the walk.
		
		\item Given a chord diagram $X \in \chorddiagrams{n}{g}$, the graph obtained by gluing together the vertices of the polygon $P$ is a unicellular map. The root edge is the polygonal side from vertex $1$ to $2$ in the polygon, and the cyclic ordering is given by the order of the edge endpoints as the polygon sides are traversed.
	\end{itemize}
	See \cref{fig:map_diagram_bijection} for an example.
	
	\begin{figure}[h!] 
		\centering
		\includegraphics[width=0.8\textwidth]{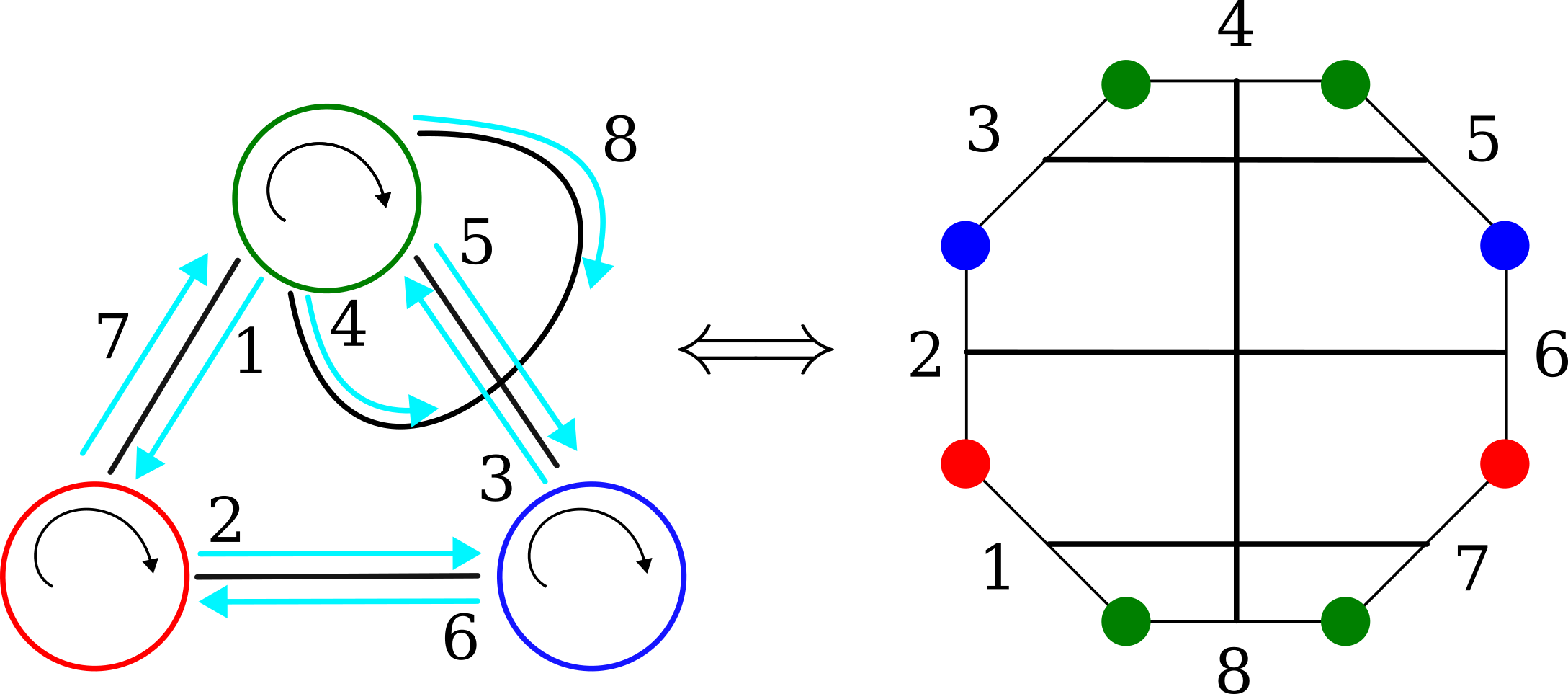}
		\captionsetup{width=.8\linewidth}
		\caption{\captionstyle{\textbf{Left}: a unicellular map with $3$ vertices and $4$ edges, including one self loop. The black arrows inside the vertices depict the cyclic ordering of the edges. Walking along its edges in order gives a single cycle (depicted in light blue) where every edge has two labels indicating the time it was traversed. \textbf{Right}: the chord diagram obtained by matching polygonal sides whose labels correspond to the same traversed edge in the map on the left.}}
		\label{fig:map_diagram_bijection}
	\end{figure}
	
	We now prove the lemma. Let $(M, e)$ be the unicellular map corresponding to $X$, let $x\in M$ be a vertex of degree $1$ or $2$, and let $M'$ be	the map obtained by deleting $x$ from $M$ in the following manner: if $x$ has degree $1$, we remove it together with its edge, and if $x$ has degree $2$, we remove $x$ and its two edges $\left(u,x\right)$, $\left(x,v\right)$ and replace them by a single edge $\left(u,v\right)$ (note that when $\deg\left(x\right)=2$, $x$ must have two distinct edges rather than be an isolated vertex with a self loop). If the root edge is deleted in this way, choose a new one arbitrarily. Denote by $X'$ the chord diagram corresponding to $M'$. Then $N(X) = N(X')$. To see this, observe that a vertex of degree $1$ in $M$ corresponds to a non-crossing chord in $X$, which cannot affect $N(X)$, while the edges of a degree-2 vertex in $M$ correspond to two parallel adjacent chords in $X$. Replacing the two edges by a single edge simply removes one of these chords in the chord diagram. Since any third chord which crosses one must also cross the other, removing one of them does not change the size of $N(X)$ (note that changing the root edge amounts only to rotating the chord diagram, and does not affect the flatly parallel sets).
	
	Let $M_{r}$ be the map obtained from $M$ by iteratively deleting vertices of degree $1$ and $2$ until all remaining vertices have degree $\geq3$. Then by the above, its chord diagram $X_r$ satisfies $N(X) = N(X_r)$.
	Further, the genus of $M_{r}$ is the same as that of $M$: whenever we delete a vertex in the manner described above, we reduce by $1$ both the number of edges and the number of vertices. The number of	vertices $N$ and edges $E$ of $M_{r}$ thus satisfy 
	\begin{equation*}
		E=2g+N-1.
	\end{equation*}	
	Since $M_{r}$ has degree at least $3$, we have $N \leq \frac{2}{3}E$,
	yielding 
	\begin{equation*}
		E\leq6g-3.
	\end{equation*}	
	The number of flatly parallel sets cannot be larger than the number of edges, so $N(X_r) \leq 6g-3$.
\end{proof}

A chord $e\in F_{i}$ is called \textit{extremal} if all other chords of $F_{i}$ lie to one side of it. If $F_{i}$ contains just a single chord, that chord is extremal; otherwise, $F_{i}$ has exactly two extremal chords (this is evident from \cref{claim:crossing_parallel_chords}). For every $i = 1,\ldots, N(X)$, let $E_i$ be the set of extremal chords of $F_i$, and let
\begin{equation*}
	S(X) = \{ E_1, \ldots, E_{N(X)}\} \, .
\end{equation*}
We call $S$ the \textit{signature} of $X$. 

Let $\parallelsigs{n}{g} = \{S(X) \mid X \in \chorddiagrams{n}{g} \}$ be the set of all signatures on chord diagrams of genus $g$ on $2n$ sides, and for each $S \in \parallelsigs{n}{g}$, let
\[
\chordsignatures{S} = \{X \in \chorddiagrams{n}{g} \mid S(X) = S \} \, 
\]
be the set of all chord diagrams that have the signature $S$. 

Let $S \in \parallelsigs{n}{g}$ be a signature, and let $E_i \in S$ be a set of extremal chords. If $E_i$ contains just a single chord, let $I_i = \emptyset$; otherwise, denoting $E_i = \{e_1, e_2\}$, let $I_i = I^{e_1, e_2}_i$ be the set of sides between $e_1$ and $e_2$, as defined in \cref{def:flatly_parallel}. By definition, for any $X \in \chordsignatures{S}$, the chords of $X$ with endpoints in $I_i$ do not cross each other. If we remove from the polygon $P$ all sides which are endpoints of extremal chords, as well as all sides belonging to any $I_i$, the remaining sides (if any) can be grouped into a collection of disjoint sets $J_1, \ldots , J_\ell$, with $\ell \leq |S|$, such that the $J_j$ are maximal under the condition that any possible chord with both endpoints in $J_j$ does not intersect an extremal chord (see \cref{fig:polygon_partition}). By definition, for any $X \in \chordsignatures{S}$, chords with endpoints in $J_i$ do not belong to any flatly parallel class, and so also do not have any crossings. Further, any chord of $X$ with one endpoint in $J_i$ must have both of its endpoints in $J_i$ (by maximality). This gives a characterization of $\chordsignatures{S}$: let $H(S) = \{I_1, \ldots, I_{|S|}, J_1, \ldots, J_\ell\}$. Then $\chordsignatures{S}$ is the set of all chord diagrams $X$ containing the extremal edges of $S$, such that the restriction of $X$ to any $A \in H(S)$ is a perfect non-crossing matching.

\begin{figure}[h!] 
	\centering
	\includegraphics[width=0.3\textwidth]{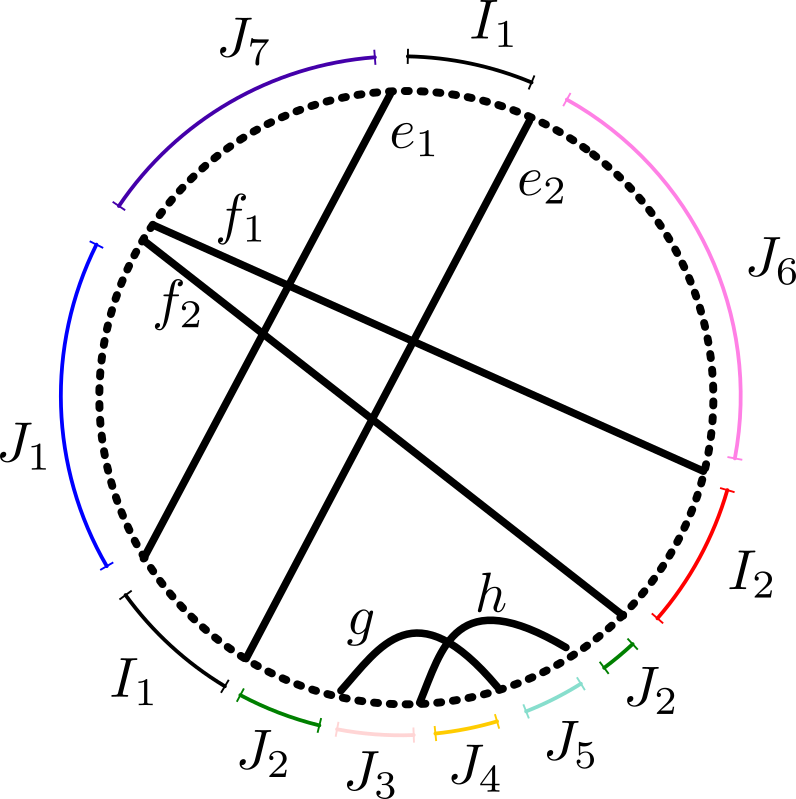}
	\captionsetup{width=.8\linewidth}
	\caption{\captionstyle{An example of the sets $I_i$ and $J_i$. There are 6 extremal edges: two pairs $e_1,e_2$ and $f_1,f_2$, and two singles, $g$ and $h$. Non-extremal chords are not shown. Note that $I_1$ is the union of two intervals (the intervals between the sides $e_1$ and $e_2$), but $I_2$ is only a single interval (since there are no sides between the neighbouring endpoints of $f_1$ and $f_2$). $J_2$ is also the union of two intervals.}}
	\label{fig:polygon_partition}
\end{figure}

\begin{proof}[Proof of \cref{thm:spectral_gap_lower_bound}]
We apply \cref{thm:projection_restriction_poincare} to the projection chain $\overline{P}$ defined according to the partition $\chorddiagrams{n}{g} = \union_{S \in \parallelsigs{n}{g}} \chordsignatures{S}$, and the restriction chains $\{P_S\}_{S \in \parallelsigs{n}{g}}$, where each $P_S$ is defined on $\{ \chordsignatures{S} \}$. \cref{thm:spectral_gap_lower_bound} follows immediately from the following two propositions, whose proofs follow next.
\begin{proposition} \label{prop:restriction_chain_gap}
	The chain $P_S$ is irreducible for every $S \in \parallelsigs{n}{g}$. There exists a universal constant $C > 0$ such that for every $S \in \parallelsigs{n}{g}$, 
	\begin{equation} \label{eq:restriction_chain_gap}
		\gamma_S \geq C n^{-4} \, .
	\end{equation}
\end{proposition}

\begin{proposition} \label{prop:projection_chain_gap}
	For every $g \geq 1$, there exists a constant $C_g$ such that 
	\begin{equation} \label{eq:projection_chain_gap}
		\overline{\gamma} \geq C_g ~ n^{-1488g+620} \, .
	\end{equation}
\end{proposition}

\end{proof}

\begin{proof}[Proof of \cref{prop:restriction_chain_gap}]
	The characterization of $\chordsignatures{S}$ as the set of chord diagrams with fixed extremal chords and non-intersecting matchings on subsets of the remaining sides shows that the restricted chain $P_S$ can be reduced to a combination of genus $0$ chains, each defined on a set of $H(S)$. Indeed, for every $A \in H(S)$, all the chords with at least one endpoint in $A$ have both endpoints in $A$, and the chords with endpoints in $A$ do not intersect each other. For every $A \in H(S)$, let $\tilde{P}_A$ be the Markov chain which at each step chooses uniformly a random pair of chords from those of $A$ and tries to swap them without creating new crossings. Then the restricted chain can be written as $P_S = \sum_{A \in H(S)} p_A \tilde{P}_A + (1-\sum_{A\in H(S)} p_A) \mathrm{Id}$, where $p_A = {\abs{A} \choose 2} / {n \choose 2}$. This description shows that $P_S$ is irreducible, as each $\tilde{P}_A$ is an irreducible genus $0$ chain. If the chains $\tilde{P}_A$ satisfy Poincar\'{e} inequalities with constants $\tilde{\gamma}_A$, then it is standard that (see e.g. \cite[Corollary 12.13]{levin_peres_markov_chains})
	\begin{equation*}
		\gamma_S \geq \min_{A\in H(S)} p_A \tilde{\gamma}_A \, .
	\end{equation*}
	By the bound \eqref{eq:cohen_gap_lower_bound}, there exists a universal constant $C > 0$ such that  $\tilde{\gamma}_A \geq C \abs{A}^{-4}$, and so 
	\begin{equation} \label{eq:restriction_chain_gap}
		\gamma_S \geq C \min_{A \in H(S)}{\abs{A} \choose 2} / {n \choose 2} \cdot \abs{A}^{-4} \geq C n^{-4} \, .
	\end{equation}
\end{proof}

\begin{proof}[Proof of \cref{prop:projection_chain_gap}]
	By the classical Cheeger inequality (see e.g. \cite[Theorem 13.10]{levin_peres_markov_chains}), the spectral gap $\overline{\gamma}$ satisfies 
	\begin{equation*}
		\overline{\gamma} \geq \frac{\Phi^2}{2} \, , 
	\end{equation*}
	where $\Phi$ is the bottleneck ratio, given by
	\begin{equation*}
		\Phi = \min_{A \subseteq \parallelsigs{n}{g}, \overline{\pi}(A) \leq 1/2} \frac{\sum_{S\in A} \sum_{T \in \parallelsigs{n}{g} \setminus A} \overline{\pi}(S)\overline{P}(S,T) }{\overline{\pi}(A)} \,.
	\end{equation*}
	The proposition will follow by showing that $\Phi > C_g n^{-744g + 310}$. 
	
	Let $A \subseteq \parallelsigs{n}{g}$. By \cref{thm:irredfixedgenus}, the fixed genus chord swap chain is irreducible, so there must exist $S \in A$ and $T \in \parallelsigs{n}{g} \setminus A$ with $\overline{P}(S,T) > 0$, and so the numerator on the right-hand side is always positive. We can then bound the bottleneck ratio by
	\begin{equation} \label{eq:bottleneck_lower_bound}
		\Phi \geq \min_{S\neq T \in \parallelsigs{n}{g}, \overline{P}(S,T)>0} 2 \overline{\pi}(S)\overline{P}(S,T) \,.
	\end{equation}

	To bound $\overline{\pi}(S)$ in \eqref{eq:bottleneck_lower_bound}, note that 
	\begin{align}
		\overline{\pi}(S) := \sum_{X \in \chordsignatures{S}} \pi(X) 
		&= \frac{\abs{\chordsignatures{S}}}{\abs{\chorddiagrams{n}{g}}} = \frac{\abs{\chordsignatures{S}}}{\sum_{T \in \parallelsigs{n}{g}} \abs{\chordsignatures{T}}} \nonumber \\
		&\geq \frac{\abs{\chordsignatures{S}}}{\abs{\parallelsigs{n}{g}}\max_{T \in \parallelsigs{n}{g}} \abs{\chordsignatures{T}}} \geq \frac{1}{\abs{\parallelsigs{n}{g}}}\min_{T \in \parallelsigs{n}{g}} \frac{\abs{\chordsignatures{S}}}{\abs{\chordsignatures{T}}} \label{eq:bounding_pi_bar} \, .
	\end{align}

	To bound the number of signatures $\abs{\parallelsigs{n}{g}}$, note that for every integer $k$, the total number of signatures with $k$ elements is bounded above by the number of ways we can place $2k$ chords (not necessarily disjoint) on a circle of $2n$ points. By \cref{lem:number_of_flatly_parallel}, $k \leq 6g-3$, and so there exists a constant $C_g$ such that 
	\begin{equation} \label{eq:num_of_signatures} 
		\abs{\parallelsigs{n}{g}} \leq  \sum_{k=1}^{6g-3} (2n)^{4k} \leq C_g \cdot n^{24g-12} \, .	
	\end{equation}

	The ratio between the number of chord diagrams of different signatures is bounded by the following claim, which is essentially a computation on ratios of Catalan numbers.

	\begin{claim}\label{claim:ratio_of_catalan}			
		Let $n\geq 1$. For $i\in \{1,2\}$, let $A_i$ be a finite set of $k_i$ chords in $P$, and let $\mathcal{J}_i$ be a partition of the sides of the polygon $P$ without the endpoints of $A_i$ into at most $2k_i$ even-sized sets. Let $\mathcal{X}_i$ be the set of all chord diagrams $X$ which contain $A_i$ and such that for every $J \in \mathcal{J}_i$, the restriction of $X$ to $J$ is a perfect non-crossing matching. Then there exists a constant $C_{k_1, k_2}$ depending only on $k_1$ and $k_2$ such that 
		\begin{equation*}
			\frac{\abs{\mathcal{X}_1}}{\abs{\mathcal{X}_2}} \geq C_{k_1, k_2} n^{-3k_2} \, .
		\end{equation*}
	\end{claim}
	
	\begin{proof}
		Denoting by $C_m$ the $m$-th Catalan number, we have by the definition of $\mathcal{X}_i$ that
		\begin{equation}
			\frac{\abs{\mathcal{X}_1}}{\abs{\mathcal{X}_2}} = \frac{\prod_{J \in \mathcal{J}_i} C_{\abs{J}/2}}{\prod_{J \in \mathcal{J}_2} C_{\abs{J}/2}} \, .
		\end{equation}
		By a well-known estimate, there exists a constant $c > 1$ such that 
		\begin{equation*}
			\frac{1}{c} \frac{4^m}{m^{3/2}} \leq C_m \leq c \frac{4^m}{m^{3/2}} \, .
		\end{equation*}
		Since $\sum_{J \in \mathcal{J}_i}\abs{J}/2 = n-k_i$ and since $\abs{\mathcal{J}_i} \leq 2k_i$, by the AM-GM inequality we have
		\begin{equation*}
			\prod_{J \in \mathcal{J}_2}\frac{\abs{J}}{2} \leq \left(\frac{n-k_2}{\abs{\mathcal{J}_2}} \right)^{\abs{\mathcal{J}_2}} \leq n^{2k_2} \, ,
		\end{equation*}
		and so 
		\begin{align*}
			\frac{\abs{\mathcal{X}_1}}{\abs{\mathcal{X}_2}}
			&\geq 
			\frac{(1/c)^{\abs{\mathcal{J}_1}}4^{\sum_{J \in \mathcal{J}_1} \abs{J}/2} ~ \prod_{J \in \mathcal{J}_1}(\abs{J}/2)^{3/2}} {c^{\abs{\mathcal{J}_2}}4^{\sum_{J \in \mathcal{J}_2} \abs{J}/2} ~\prod_{J \in \mathcal{J}_2}(\abs{J}/2)^{3/2}} \\
			&\geq
			c^{-(\abs{\mathcal{J}_1} + \abs{\mathcal{J}_2})}4^{\sum_{J \in \mathcal{J}_1}\abs{J}/2-\sum_{J \in \mathcal{J}_2}\abs{J}/2} \prod_{J \in \mathcal{J}_2}\left(\abs{J}/2 \right)^{-3/2} \\
			&\geq 
			c^{-(2k_1 + 2k_2)} 4^{k_2 - k_1} n^{-3k_2} \, . \qedhere
		\end{align*}		
	\end{proof}
	The signatures classes $\chordsignatures{S}$ and $\chordsignatures{T}$ in \eqref{eq:bounding_pi_bar} are exactly of the form required by \cref{claim:ratio_of_catalan}, with $k_i \leq 12g -6$ (due to \cref{lem:number_of_flatly_parallel}). There thus exists a constant $C_g$ such that  
	\begin{equation}\label{eq:projected_stationary_is_big_enough}
		\overline{\pi}(S)
		\geq 
		\frac{1}{\abs{\parallelsigs{n}{g}}}\min_{T \in \parallelsigs{n}{g}} \frac{\abs{\chordsignatures{S}}}{\abs{\chordsignatures{T}}} 
		\geq 
		C_g n^{-24g + 12} n^{-36g + 18} = C_g n^{-60g + 30}
		\, .
	\end{equation}
	
	To bound the transition probability in \eqref{eq:bottleneck_lower_bound}, note that for any pair of signatures $S$ and $T$, we have 
	\begin{equation*}
		\overline{P}(S,T) := \frac{1}{\overline{\pi}(S)} \sum_{\substack{X \in \chordsignatures{S} \\ Y \in \chordsignatures{T}}} \pi(x) P(X,Y) = \frac{1}{\sum_{X \in S} \pi(X) } \sum_{\substack{X \in \chordsignatures{S} \\ Y \in \chordsignatures{T}}} \pi(X) P(X,Y) \, .
	\end{equation*}
	Since $\pi(X)$ is the uniform distribution over $\chorddiagrams{n}{g}$ and $P(X,Y) = 1/(2{n \choose 2})$ if it is not zero, we get 
	\begin{equation*}
		\overline{P}(S,T) = \frac{1}{\abs{\chordsignatures{S}}} \sum_{\substack{X \in \chordsignatures{S} \\ Y \in \chordsignatures{T}}} \frac{\mathbf{1}_{P(X,Y) \neq 0} }{2 {n \choose 2}} \, .
	\end{equation*}

	We now show that there are sufficiently many pairs $(X,Y) \in \chordsignatures{S} \times \chordsignatures{T}$ with non-zero transition probability. 
	\begin{claim}\label{claim:enough_to_project}
		There exists a constant $C_g$ depending only on $g$ such that for every $S,T \in \chorddiagrams{n}{g}$ with $\overline{P}(S,T)>0$, there exists a subset $\mathcal{X}' \subseteq \chordsignatures{S}$ of size				
		\begin{equation*}
			\abs{\mathcal{X}'} \geq C_g n^{-684g + 282} \abs{\chordsignatures{T}}\, ,
		\end{equation*}
		such that for every $X' \in \mathcal{X}'$ there exists a $Y' \in \chordsignatures{T}$ that is a obtained from $X'$ by a chord swap.		
	\end{claim}	
	\begin{proof}
		Let $X \in \chordsignatures{S}$ and let $Y \in \chordsignatures{T}$ be chord diagrams such that there exist chords $e_1,e_2 \in X$ and $f_1,f_2 \in Y$ such that $Y$ is obtained from $X$ by the swap $(e_1, e_2) \mapsto (f_1, f_2)$. Let $E_X$ be the set of extremal chords of $X$ , let $E_Y$ the set of extremal chords of $Y$, and let
		\[
		 	E_0 = E_X \union \left(E_Y \cap X \right) \union \{e_1, e_2\} \, .
		\]
		If we remove from $P$ all endpoints of chords in $E_0$, the remaining sides can be partitioned into maximally long disjoint intervals $J_1, \ldots J_k$; by \cref{lem:number_of_flatly_parallel}, 
		\begin{equation} \label{eq:num_of_intial_chords_to_remove}
			k \leq 2\abs{E_0} \leq 2(\abs{E_X} + \abs{E_Y \cap X} + 2) \leq (12g-6) + (12g-6) + 4 = 24g-8 \, .
		\end{equation}		
		Ideally, we would have liked to set $\mathcal{X}'$ to be the set of all chord diagrams in $\chordsignatures{S}$ which contain the chords $E_0$, and whose restriction to the interval $J_i$ is a perfect non-crossing matching on $J_i$ for every $i=1,\ldots,k$. However, since some of the intervals $J_i$ may have odd size, this is not always possible. This may be remedied by adding more chords to $E_0$, as follows. For every $1\leq i \neq j \leq k$, let $C_{ij}$ be the set of chords in $X$ between $J_i$ and $J_j$. The chords of $C_{ij}$ may be enumerated according to the order in which they are encountered traversing the polygon along the sides of $J_i$ in a clockwise manner. If $\abs{C_{ij}}$ is odd, let $e_{ij}^f$, $e_{ij}^c$ and $e_{ij}^\ell$ be its first, centre, and last chords, respectively. For every $i \neq j$, denote 
		\begin{equation*}
			E_{ij} = 
			\begin{cases}
				\{e_{ij}^c\} & \abs{C_{ij}} = 2m+1,\text{ $m$ even} \\
				\{e_{ij}^c, e_{ij}^f, e_{ij}^\ell\} & \abs{C_{ij}} = 2m+1, \text{ $m$ odd} \\
				\emptyset & \abs{C_{ij}}\text{ even} \, .
			\end{cases}
		\end{equation*}
		Finally, set
		\begin{equation*}
			E_1 = E_0 \union \bigcup_{1 \leq i \neq j \leq k}E_{ij} \, .
		\end{equation*}  
		If we remove from $P$ the endpoints of the chords in $E_1$, the remaining sides can be partitioned into maximally long disjoint intervals $I_1, \ldots I_\ell$, obtained by removing from each $J_i$ the endpoints of chords in $\union_{j\neq i}E_{ij}$. We now have the crucial fact that for every $s = 1,\ldots,\ell$, $\abs{I_s}$ is even. To see this, consider the interval $J_i$, and observe that removing $E_{ij}$ from $C_{ij}$ results in splitting $C_{ij}$ into two even parts $C_{ij}^1$ an $C_{ij}^2$. If we remove the endpoints of the chords of $C_{ij}$ from $J_i$, we obtain disjoint even-sized intervals, and restricting $X$ to every such interval gives a perfect non-crossing matching. Each $I_s$ is therefore a union of a collection of these even-sized intervals, a collection of the even-sized $C_{ij}$, and possibly of some even-sized $C_{ij}^1$ or $C_{ij}^2$ (see \cref{fig:odds_and_evens}).
		
		\begin{figure}[h!] 
			\centering
			\includegraphics[width=0.9\textwidth]{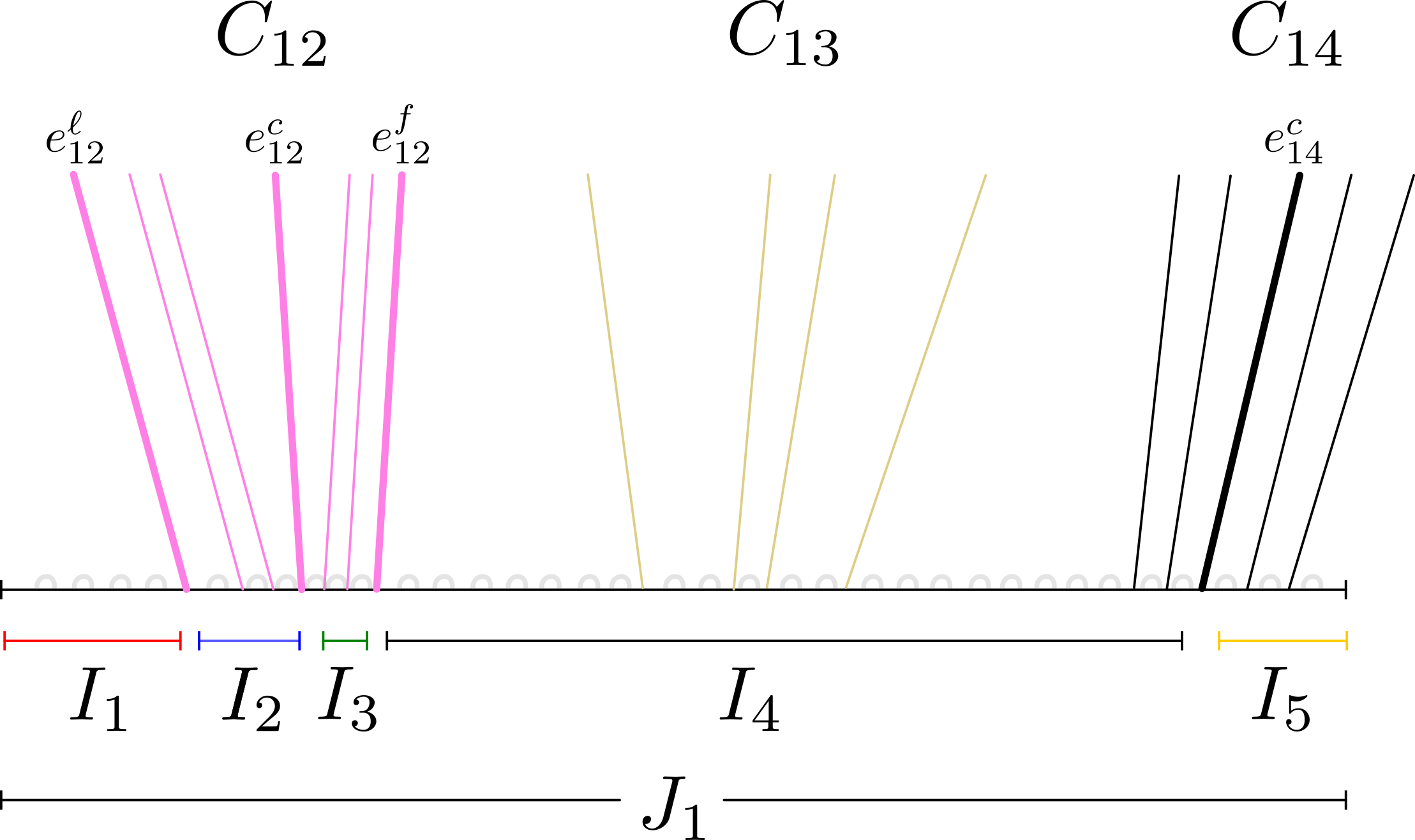}
			\captionsetup{width=.8\linewidth}
			\caption{\captionstyle{An example of how to ensure that that the intervals $I_s$ have even size. The interval $J_1$ is subdivided into $5$ intervals by removing the chords $e_{12}^\ell$, $e_{12}^c$, $e_{12}^f$, and $e_{14}^c$; after removing these chords, the sets $C_{12}$ and  $C_{14}$ have an even number of chords remaining. $C_{13}$ does not need to have any chords removed, since it already has even size. Grey chords represent arbitrary non-intersecting chords, and have an even number of sides. }}
			\label{fig:odds_and_evens}
		\end{figure}
		
		To bound $\abs{E_1}$, observe that the chords in different $E_{ij}$ do not cross each other. Thus, the graph $G$ whose vertices are the intervals $J_i$ ,and with an edge between $J_i$ and $J_j$ if and only if $C_{ij}$ is odd, is planar (the polygon $P$ itself gives an embedding). The number of non-empty sets $E_{ij}$ is therefore bounded by $3k-6$ (the maximum number of edges in a planar graph with $k$ vertices), and the total number of chords added to $E_0$ is bounded by $9k-18$; by \eqref{eq:num_of_intial_chords_to_remove}, we therefore have
		\begin{equation*}
			\abs{E_1} \leq \abs{E_0} + 9k-18 \leq 12g-4 + 9(24g-8)-18 = 228g-94 \, .
		\end{equation*}
		
		Let $\mathcal{X}'$ be the set of of all chord diagrams in $\chordsignatures{S}$ which contain the chords $E_1$, and whose restriction to the interval $I_i$ is a perfect non-crossing matching for every $i=1,\ldots, \ell$. Then for every $X' \in \mathcal{X}'$, the swap $(e_1, e_2) \mapsto (f_1, f_2)$ yields a $Y' \in \chordsignatures{T}$. Indeed, $X'$ has the same signature as $X$, and the chords on the intervals $I_i$ are non-intersecting and cannot be extremal in either $X$ or $Y$ and so cannot affect the signature. Finally, by \cref{claim:ratio_of_catalan}, there exists a constant $C_g$ such that 
		\begin{equation*}
			\frac{\abs{\mathcal{X}'}}{\abs{\chordsignatures{S}}} \geq C_g n^{-3\abs{E_1}} \geq C_g n^{-684g + 282}\, .
		\end{equation*}		
	\end{proof}			
	With \cref{claim:enough_to_project} in hand, we have 
	\begin{equation}\label{eq:projected_transition_is_big_enough}
		\overline{P}(S,T) = \frac{1}{\abs{\chordsignatures{S}}} \sum_{\substack{X \in \chordsignatures{S} \\ Y \in \chordsignatures{T}}} \frac{\mathbf{1}_{P(X,Y) \neq 0} }{2 {n \choose 2}} 
		\geq 
		\frac{1}{n^2} \frac{1}{\abs{\chordsignatures{S}}} \sum_{X' \in \mathcal{X}'} 1 
		=
		\frac{1}{n^2} \frac{\abs{\mathcal{X}'}}{\abs{\chordsignatures{S}}} 
		\geq C_g n^{-684g + 280}
		\, ,
	\end{equation}
	Plugging in \eqref{eq:projected_stationary_is_big_enough} and \eqref{eq:projected_transition_is_big_enough} into the bottleneck ratio \eqref{eq:bottleneck_lower_bound}, we get 
	\begin{equation*}
		\Phi \geq C_g n^{-744g + 310} \, 
	\end{equation*}
	as needed.

\end{proof}

\subsection*{Acknowledgements}
We thank Alessandra Caraceni and Alexandre Stauffer for helpful comments and discussions.

\bibliographystyle{amsalpha}
\bibliography{chord_swapping_bibliography}

\end{document}